\def\todaysdate{2 October 2020}
\documentclass[reqno,a4paper]{article}
\usepackage{etex}
\usepackage[T1]{fontenc}
\usepackage[utf8]{inputenc}
\usepackage{lmodern}

\usepackage[style=alphabetic,firstinits=true,backref,backend=biber,isbn=false,url=false,maxbibnames=99]{biblatex}

\renewbibmacro{in:}{}
\newbibmacro{string+doi}[1]{\iffieldundef{doi}{\iffieldundef{url}{#1}{\href{\thefield{url}}{#1}}}{\href{http://dx.doi.org/\thefield{doi}}{#1}}}
\DeclareFieldFormat*{title}{\usebibmacro{string+doi}{\emph{#1}}}
\DefineBibliographyStrings{english}{%
  backrefpage = {$\uparrow$},
  backrefpages = {$\uparrow$},
}

\usepackage{amssymb,amsmath,amsthm}
\usepackage{units}
\usepackage{graphicx}
\usepackage[all]{xy}
\usepackage{tikz}
\usetikzlibrary{arrows,calc,positioning,decorations.pathreplacing}
\usepackage{relsize}
\usepackage{mhsetup}
\usepackage{mathtools}
\usepackage{stmaryrd}
\usepackage{tocloft}
\usepackage{paralist} 
\usepackage[left=3cm,right=3cm,top=3cm,bottom=3cm]{geometry} 
\usepackage[bottom]{footmisc}
\usepackage[colorlinks,citecolor=blue,linkcolor=blue,urlcolor=blue,filecolor=blue,bookmarksdepth=2,breaklinks]{hyperref}
\usepackage{footnotebackref}
\usepackage[format=plain,indention=1em,labelsep=quad,labelfont={up},textfont={footnotesize},margin=2em]{caption}
\usepackage[mathscr]{euscript}
\usepackage{accents}

\definecolor{lightblue}{rgb}{0.8,0.8,1}

\numberwithin{equation}{section}

\newtheoremstyle{italicised}
        {\topsep}{\topsep}  
        {\itshape}  
        {}  
        {\bfseries}  
        {}  
        {1ex}  
        {}  
\theoremstyle{italicised}
\newtheorem{thm}{Theorem}[section]
\newtheorem{lem}[thm]{Lemma}
\newtheorem{prop}[thm]{Proposition}
\newtheorem{coro}[thm]{Corollary}
\newtheorem{athm}{Theorem}
\newtheorem{acoro}[athm]{Corollary}

\newtheoremstyle{upright}
        {\topsep}{\topsep}  
        {\upshape}  
        {}  
        {\bfseries}  
        {}  
        {1ex}  
        {}  
\theoremstyle{upright}
\newtheorem{defn}[thm]{Definition}
\newtheorem{rmk}[thm]{Remark}
\newtheorem{eg}[thm]{Example}
\newtheorem{notation}[thm]{Notation}
\newtheorem{construction}[thm]{Construction}
\newtheorem{inputdata}[thm]{Input}
\newtheorem*{note}{Note}

\newtheoremstyle{italicised-restate}
        {\topsep}{\topsep}  
        {\itshape}  
        {}  
        {\bfseries}  
        {}  
        {1ex}  
        {\thmname{#1}\thmnote{ \bfseries #3}}  
\theoremstyle{italicised-restate}
\newtheorem{rthm}{Theorem}
\newtheorem{rcoro}{Corollary}

\makeatletter
\renewcommand*{\@seccntformat}[1]{\upshape\csname the#1\endcsname.\hspace{1ex}}
\renewcommand*{\section}{\@startsection{section}{1}{\z@}%
	{2.5ex \@plus 1ex \@minus 0.2ex}%
	{1.5ex \@plus 0.2ex}%
	{\normalfont\large\bfseries}}
\renewcommand*{\subsection}{\@startsection{subsection}{2}{\z@}%
	{2.5ex \@plus 1ex \@minus 0.2ex}%
	{1.5ex \@plus 0.2ex}%
	{\normalfont\normalsize\bfseries}}
\renewcommand*{\subsubsection}{\@startsection{subsubsection}{3}{\z@}%
	{2.5ex \@plus 1ex \@minus 0.2ex}%
	{-1.5ex \@plus -0.2ex}%
	{\normalfont\normalsize\bfseries}}
\renewcommand*{\paragraph}{\@startsection{paragraph}{4}{\z@}%
	{2.5ex \@plus 1ex \@minus 0.2ex}%
	{-1.5ex \@plus -0.2ex}%
	{\normalfont\normalsize\bfseries}}
\renewcommand*{\subparagraph}{\@startsection{subparagraph}{5}{\z@}%
	{2.5ex \@plus 1ex \@minus 0.2ex}%
	{-1.5ex \@plus -0.2ex}%
	{\normalfont\normalsize\slshape}}
\makeatother

\newcommand{\lhto}{\lhook\joinrel\longrightarrow}

\newcommand{\cf}{\textit{cf}.\ }

\newcommand{\incl}[3][right]%
{%
\draw[<-,>=#1 hook] #2 to ($ #2!0.5!#3 $);
\draw[->] ($ #2!0.5!#3 $) to #3;%
}
\newcommand{\inclusion}[5][right]%
{%
\draw[<-,>=#1 hook] #4 to ($ #4!0.5!#5 $) node[#2,font=\small]{#3};
\draw[->] ($ #4!0.5!#5 $) to #5;%
}

\newenvironment{itemizeb}%
{\begin{compactitem}

}%
{\end{compactitem}}

\newcommand{\cB}{\mathcal{B}}

\newcommand{\cP}{\mathcal{P}}

\newcommand{\cR}{\mathcal{R}}

\newcommand{\bD}{\mathbb{D}}

\newcommand{\bP}{\mathbb{P}}
\newcommand{\bQ}{\mathbb{Q}}
\newcommand{\bR}{\mathbb{R}}

\newcommand{\bZ}{\mathbb{Z}}

\renewcommand{\geq}{\geqslant}
\renewcommand{\leq}{\leqslant}
\renewcommand{\footnoterule}{%
  \kern -3pt
  \hrule width \textwidth height 0.4pt
  \kern 2.6pt
}

\newlength{\dhatheight}
\newcommand{\doublehat}[1]{%
    \settoheight{\dhatheight}{\ensuremath{\hat{#1}}}%
    \addtolength{\dhatheight}{-0.3ex}%
    \hat{\vphantom{\rule{1pt}{\dhatheight}}%
    \smash{\hat{#1}}}}
\newlength{\dbarheight}
\newcommand{\doublebar}[1]{%
    \settoheight{\dbarheight}{\ensuremath{\bar{#1}}}%
    \addtolength{\dbarheight}{-0.15ex}%
    \bar{\vphantom{\rule{1pt}{\dbarheight}}%
    \smash{\bar{#1}}}}

\newcommand{\rinftyz}{\ensuremath{\bR^{\infty,0}}}
\newcommand{\rinftym}{\ensuremath{\bR^{\infty,-1}}}
\newcommand{\rinftymt}{\ensuremath{\bR^{\infty,-2}}}
\newcommand{\colhat}{\ensuremath{\mathrm{c}\hat{\mathrm{o}}\mathrm{l}}}
\newcommand{\coldoublehat}{\ensuremath{\mathrm{c}\doublehat{\mathrm{o}}\mathrm{l}}}

\usepackage{fancyhdr}
\usepackage{lastpage}
\fancypagestyle{plain}{%
\lhead{}
\chead{}
\rhead{}
\lfoot{}
\cfoot{p.\ \thepage\ $/\!\!/$ \pageref*{LastPage}}
\rfoot{}

}

\begin{document}
\title{\large\bfseries Stability for moduli spaces of manifolds with conical singularities \vspace{-1ex}}
\author{\small Martin Palmer\quad $/\!\!/$\quad \todaysdate\vspace{-1ex}}
\date{}
\maketitle
{
\makeatletter
\renewcommand*{\BHFN@OldMakefntext}{}
\makeatother
\footnotetext{2010 \textit{Mathematics Subject Classification}: 55R40, 55R80, 57N20, 57S05, 58B05.}
\footnotetext{\textit{Key words}: Moduli spaces of submanifolds, diffeomorphism groups, configuration spaces, homological stability, twisted coefficients, polynomial functors, manifolds with singularities, parametric connected sum.}
}
\begin{abstract}
The homology of configuration spaces of point-particles in manifolds has been studied intensively since the 1970s; in particular it is known to be \emph{stable} if the underlying manifold is connected and open. Closely related to configuration spaces are \emph{moduli spaces of manifolds with marked points}, and in \cite{Tillmann2016Homologystabilitysymmetric} this relation was used to show that the homology of these moduli spaces is also stable as the number of marked points goes to infinity.

Since a disc neighbourhood of a marked point may be viewed as the cone on its boundary sphere, we may think of marked points as inessential conical singularities in the underlying manifold. In this paper, we prove stability for the homology of \emph{moduli spaces of manifolds with conical singularities}, allowing more general -- \emph{essential} -- singularities.

This is deduced as a special case of a more general homological stability result for classifying spaces of \emph{symmetric diffeomorphism groups} of manifolds, with respect to \emph{parametric connected sum}, an operation generalising ordinary connected sum and surgery (including Dehn surgery).

The key input for the proof is stability for the homology of \emph{moduli spaces of disconnected submanifolds} \cite{Palmer2018HomologicalstabilitymoduliI}, which generalises configurations of point-particles to configurations of higher-dimensional submanifolds. Along the way, we extend this result further to homology with polynomially twisted coefficients and configurations of labelled submanifolds.
\end{abstract}


\section{Introduction}

Let $M$ be a connected manifold with $\partial M \neq \varnothing$. As the number of points $n\to\infty$, the homology of the unordered configuration spaces $C_n(M)$ \emph{stablises}: there are maps $C_n(M) \to C_{n+1}(M)$ inducing isomorphisms on integral homology in a diverging range of degrees. This was proved first for the plane $M = \bR^2$ by Arnol'd~\cite{Arnold1970Certaintopologicalinvariants} and in general by McDuff and Segal~\cite{Segal1973Configurationspacesand,McDuff1975Configurationspacesof,Segal1979topologyofspaces}. This has been extended to \emph{labelled} configuration spaces $C_n(M;X)$, where the configuration points are equipped with labels in a path-connected space $X$ \cite{Randal-Williams2013Homologicalstabilityunordered}, as well as homology with \emph{polynomially twisted} coefficients \cite{Palmer2018Twistedhomologicalstability}.

Closely related, there are \emph{moduli spaces of manifolds with marked points}. Formally, these are the classifying spaces $B\mathrm{Diff}(M,\{x_1,\ldots,x_n\})$ of the groups of diffeomorphisms of $M$ fixing a given set of $n$ points in its interior, and they may be realised more concretely as the spaces of submanifolds of $\bR^\infty$ that are diffeomorphic to $M$ and equipped with an $n$-point configuration in their interiors. Stability for the homology of these moduli spaces was proven in \cite{Tillmann2016Homologystabilitysymmetric} using their connection to labelled configuration spaces on $M$.

We may think of marked points in $M$ as \emph{conical singularities} by writing a disc neighbourhood of each marked point as the cone on its boundary sphere -- see the left-hand side of Figure \ref{fig:conical-singularities} on the next page. However, the cone on a sphere is still a manifold, so these singularities are ``inessential''. More general conical singularities of an $m$-dimensional manifold are points where the manifold is locally homeomorphic to the cone on an $(m-1)$-dimensional manifold $L$, the \emph{link} of the singularity. Hence it, in fact, fails to be a manifold at these points unless $L \cong S^{m-1}$. See the right-hand side of Figure \ref{fig:conical-singularities} for an illustration of a conical singularity whose link is a $2$-torus.

In this paper we prove stability, as the number of singularities goes to infinity, for the homology of \emph{moduli spaces of manifolds with conical $L$-singularities}, where $L = \partial T$ is the boundary of a tubular neighbourhood of a closed submanifold $P \subset M$ of sufficiently high codimension. These are the classifying spaces $B\mathrm{Diff}^L(\mathbf{M}_n)$ of the $L$-diffeomorphism groups of the singular manifolds $\mathbf{M}_n$ obtained by collapsing $n$ tubular neighbourhoods of copies of $P$ down to $n$ conical singularities. Note that taking $P = \text{point}$ recovers the case of marked points (inessential singularities).

This will turn out to be a special case of our main result: stability for the homology of the classifying spaces of
\begin{equation}
\label{eq:sigmadiff}
\Sigma\mathrm{Diff}(M \sharp_P N \sharp_P \cdots \sharp_P N),
\end{equation}
where $\sharp_P$ denotes \emph{parametric connected sum} along a submanifold $P$ and the \emph{symmetric diffeomorphism group} $\Sigma\mathrm{Diff}$ is, roughly, the group of diffeomorphisms that preserve this iterated parametric connected sum decomposition. Parametric connected sum is a natural generalisation of connected sum, and one may encode surgery and, in dimension $3$, Dehn surgery as the operation $- \sharp_P N$ by taking $P$ and $N$ to be appropriate spheres or lens spaces.

The essential input for our proofs, as in the case of marked points, is stability for the homology of labelled configuration spaces -- except that we consider \emph{configurations of disconnected submanifolds} of $M$ rather than points, and the labels of the components take values in a bundle over a space of embeddings into $M$. Without labels (and with a bound on the dimension of the submanifolds), this was proven by the author in \cite{Palmer2018HomologicalstabilitymoduliI}. In order to apply it to prove stability for \eqref{eq:sigmadiff} (and thus for moduli spaces of manifolds with conical singularities), we extend this stability result to configurations of disconnected submanifolds \emph{with labels}. To do this, we first extend it to homology with \emph{polynomially twisted coefficients} -- a higher-dimensional analogue of \cite{Palmer2018Twistedhomologicalstability}.

A diagram illustrating this sequence of implications, starting from \cite{Palmer2018HomologicalstabilitymoduliI} and \cite{Palmer2018Twistedhomologicalstability}, is given on page \pageref{diagram-implications} below.

In the remainder of this introduction, we describe parametric connected sum in more detail (\S\ref{ss:parametric-connected-sum}), state precisely our main stability results (Theorems \ref{tmain}, \ref{tlabelled}, \ref{ttwisted} and Corollary \ref{coro-singularities}) and then briefly discuss the stable homology (\S\ref{ss:stable-homology}) and analogues of our results for mapping class groups (\S\ref{ss:mcg}).

\begin{figure}[t]
\centering
\includegraphics[scale=1]{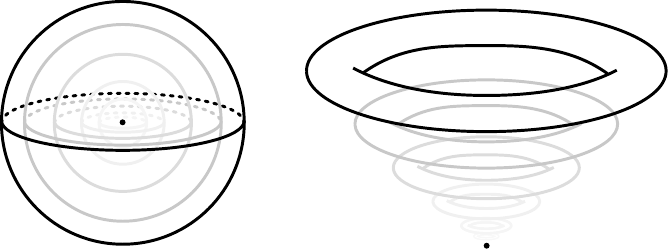}
\caption{A marked point viewed as an inessential singularity given by the cone on a codimension-$1$ sphere; an essential singularity given by the cone on a torus.}
\label{fig:conical-singularities}
\end{figure}

\subsection{Parametric connected sum}
\label{ss:parametric-connected-sum}

\begin{defn}\label{d:param-connected-sum}
Given two smooth embeddings $f \colon L \hookrightarrow M$ and $g \colon L \hookrightarrow N$ with isomorphic normal bundles, we may form the \emph{parametric connected sum} of $M$ and $N$ along $L$ as follows. It depends also on the choice of a bundle isomorphism $\theta \colon \nu(f) \cong \nu(g)$ between the two normal bundles.

First choose a metric on the vector bundle $\nu(f) \to L$ and tubular neighbourhoods $\bar{f} \colon \nu(f) \hookrightarrow M$ and $\bar{g} \colon \nu(g) \hookrightarrow N$. (This is a contractible choice.) For a vector bundle $E$ with a metric and an interval $I \subseteq [0,\infty)$, write $E_I$ for the subbundle consisting of all vectors with norm in $I$.

The parametric connected sum $M \sharp_L N$ is formed by removing the neighbourhood $\bar{f}(\nu(f)_{[0,1]})$ of $f(L)$ from $M$ and the neighbourhood $\bar{g}(\nu(g)_{[0,1]})$ of $g(L)$ from $N$ and then gluing together the slightly larger neighbourhoods $\bar{f}(\nu(f)_{(1,2)})$ and $\bar{g}(\nu(g)_{(1,2)})$ by $\theta \circ \sigma$, where $\sigma$ is the involution of $\nu(f)_{(1,2)}$ that acts radially on each fibre, sending vectors of norm $r$ to vectors of norm $3-r$. This may be written as a pushout diagram:
\begin{center}
\begin{tikzpicture}
[x=1mm,y=1mm]
\node (l) at (0,0) {$\nu(f)_{(1,2)}$};
\node (t) at (40,8) {$M \smallsetminus \bar{f}(\nu(f)_{[0,1]})$};
\node (b) at (40,-8) {$N \smallsetminus \bar{g}(\nu(g)_{[0,1]})$};
\node (r) at (80,0) {$M \sharp_L N$};
\inclusion{above}{$\bar{f} \circ \sigma$}{(l.10)}{(t.west)}
\inclusion{below}{$\bar{g} \circ \theta$}{(l.350)}{(b.west)}
\incl{(t.east)}{(r.170)}
\incl{(b.east)}{(r.190)}
\end{tikzpicture}
\end{center}
When $L$ is a point, this recovers the usual definition of connected sum of two manifolds.
\end{defn}

\begin{eg}\label{eg:surgery}
Take $g$ to be the standard inclusion of $S^p$ into $S^m$, for $p<m$ and let $M$ have dimension $m$. Since $\nu(g)$ is trivial, an embedding $f \colon S^p \hookrightarrow M$ together with a bundle isomorphism $\theta$ as in Definition \ref{d:param-connected-sum} is the same thing as a \emph{framed} embedding. The parametric connected sum of $M$ with $S^m$ along $S^p$ is the result of $p$-surgery along this framed embedding.
\end{eg}

\begin{eg}\label{eg:Dehn-surgery}
Take $N$ to be the lens space $L(p,q)$, thought of as the union of two solid tori along an appropriate identification of their boundaries, and take $g \colon S^1 \hookrightarrow L(p,q)$ to be the inclusion of the core of one of these solid tori. As in the previous example, $\nu(g)$ is trivial. Given a framed knot in a $3$-manifold $M$, the parametric connected sum of $M$ with $L(p,q)$ along this knot is the result of Dehn surgery of slope $p/q$ along this knot.
\end{eg}

Apart from these ubiquitous examples of surgery and Dehn surgery, instances of the more general parametric connected sum have also appeared numerous times before in the literature. For example, the notion of ``connected sum along a $k$-skeleton'' was used by Kreck~\cite[pp.\ 25--26]{Kreck1985Surgery} (see also \cite[\S 5 and \S 6]{Kreck2016Classificationofmanifolds}) to give a geometric definition of the group structure on the set of \emph{$2n$-manifolds with normal $(n-1)$-smoothing} up to stable diffeomorphism. This uses Wall's theory of \emph{thickenings} \cite{Wall1966Thickenings}, which allows one, under certain conditions, to approximate a map from a CW-complex to a manifold $M$ by a homotopy equivalence onto a compact, codimension-zero submanifold of $M$; these approximations play a role analogous to that of the tubular neighbourhoods in Definition \ref{d:param-connected-sum} above. A version of the parametric connected sum in a context where all manifolds are equipped with embeddings into a fixed Euclidean space was defined in \cite[page 264]{Skopenkov2007newinvariantparametric}; this is also where the name \emph{parametric connected sum} seems to have been first introduced. More details and further references may be found at \cite{MAP-Parametricconnectedsum}.

\subsection{Results}\label{ss:results}

Let $M$ be a connected $m$-manifold with boundary and $P$ a closed $p$-manifold. Choose embeddings $P \hookrightarrow \partial M$ and $P \hookrightarrow N$, where $N$ is another compact $m$-manifold. Assume that the normal bundles of these embeddings (after pushing the first one into the interior of $M$) are isomorphic, and choose an isomorphism between them. We may then define
\[
M \underset{nP}{\sharp} nN
\]
to be the result of performing $n$ parametric connected sum operations, each with a different copy of $N$, along pairwise disjoint embeddings of $P$ into a collar neighbourhood of $M$, parallel to the chosen embedding into its boundary. (See Definition \ref{d:param-conn-sum}.)

\begin{defn}[{\itshape Informal}]
A diffeomorphism of this manifold is called \emph{symmetric} if it preserves, setwise, the decomposition into (a) $M$ minus $n$ tubular neighbourhoods of $P$ and (b) $n$ copies of $N$ minus a tubular neighbourhood of $P$. Moreover, on the intersection of (a) and (b), it must act by fibrewise diffeomorphisms of the tubular neighbourhood (and permutations). We also require it to act by the identity on a neighbourhood of $\partial M$.

Write $\mathrm{Diff}_{\mathrm{fib}}(T)$ for the group of fibrewise diffeomorphisms of the tubular neighbourhood (normal bundle) $T \to P$, namely those diffeomorphisms of $T$ that cover some diffeomorphism of $P$. We may more generally fix a subgroup $H \leq \mathrm{Diff}_{\mathrm{fib}}(T)$ and say that a diffeomorphism is \emph{$H$-symmetric} if it acts by $H$ (and permutations) on the intersection of (a) and (b). We write
\[
\Sigma_H \mathrm{Diff}(M \underset{nP}{\sharp} nN)
\]
for the group of $H$-symmetric diffeomorphisms, and we omit the subscript $H$ if it is the whole group $\mathrm{Diff}_{\mathrm{fib}}(T)$. See Definition \ref{d:symm-diff-group} for the precise definition.
\end{defn}

Since $M$ has boundary, and symmetric diffeomorphisms are required to fix a neighbourhood of it pointwise, there is an inclusion
\[
M \underset{nP}{\sharp} nN \lhto M \underset{(n+1)P}{\sharp} (n+1)N,
\]
given by extending the collar neighbourhood of $\partial M$ and performing another parametric connected sum with $N$ along an embedding of $P$ into this extension, and a homomorphism
\begin{equation}\label{eq:stab-hom}
\Sigma_H \mathrm{Diff}(M \underset{nP}{\sharp} nN) \longrightarrow \Sigma_H \mathrm{Diff}(M \underset{(n+1)P}{\sharp} (n+1)N),
\end{equation}
given by extending symmetric diffeomorphisms by the identity. Our main result will hold under the following hypotheses (see also \S\ref{s:hs}):

\begin{itemizeb}
\item[(a)] The dimensions $p$ and $m$ satisfy $p \leq \tfrac12(m-3)$.
\item[(b)] Let $z \colon \mathrm{Diff}_{\mathrm{fib}}(T) \to \mathrm{Diff}(P)$ be the homomorphism that sends a fibrewise diffeomorphism to the diffeomorphism of $P$ that it covers. Then $z(H)$ is open in $\mathrm{Diff}(P)$ and $H \cap \mathrm{ker}(z)$ is closed in $\mathrm{ker}(z)$. Moreover, we assume that the coset space $\mathrm{ker}(z)/(H \cap \mathrm{ker}(z))$ is path-connected.
\item[(c)] Every diffeomorphism in $H$ (thought of as a diffeomorphism of a tubular neighbourhood of $P$ in $N$) extends to the whole manifold $N$.
\end{itemizeb}

We will discuss these hypotheses more in \S\ref{ss:hypotheses}, but first we state the main results of the paper.

\begin{rthm}[\ref{tmain}] \textup{[page \pageref{tmain}]}
Under these hypotheses, the homomorphism \eqref{eq:stab-hom} induces isomorphisms on homology up to degree $\tfrac{n}{2} - 1$ and split-injections in all degrees.
\end{rthm}

\begin{note}
When we speak of the \emph{homology} of a diffeomorphism group, we will always mean the homology of its classifying space (as a topological group, equipped with the Whitney $C^\infty$ topology).
\end{note}

In \S\ref{s:conical-singularities} we reinterpret a special case of this as a result about diffeomorphism groups of manifolds with conical singularities. Write $\partial T$ for the boundary of the disc bundle $T \to P$, in other words, the unit sphere bundle of the normal bundle of our chosen embedding of $P$ into $M$. The group of fibrewise diffeomorphisms of $T$ is naturally a subgroup of $\mathrm{Diff}(\partial T)$, and therefore so is $H$. In \S\ref{ss:rel-symm-diff} we construct a sequence $\mathbf{M}_n$ of manifolds with conical singularities of ``type'' $\partial T$ by collapsing tubular neighbourhoods of embedded copies of $P$ in $M$. There is an inclusion $\mathbf{M}_n \hookrightarrow \mathbf{M}_{n+1}$, which may be thought of as adjoining a new singularity of type $\partial T$, and a homomorphism
\begin{equation}\label{eq:stab-hom-sing}
\mathrm{Diff}_H^{\partial T}(\mathbf{M}_n) \longrightarrow \mathrm{Diff}_H^{\partial T}(\mathbf{M}_{n+1})
\end{equation}
given by extending diffeomorphisms by the identity. The notation $\mathrm{Diff}_H^{\partial T}(\phantom{-})$ means the group of diffeomorphisms of a manifold with conical $\partial T$-singularities that act by the cone on $H$ near each singularity. See Definition \ref{d:diff-manifold-with-conical-sing} for more precise details.

\begin{rcoro}[\ref{coro-singularities}] \textup{[page \pageref{coro-singularities}]}
Under the hypotheses \textup{(a)} and \textup{(b)} above, the homomorphism \eqref{eq:stab-hom-sing} induces isomorphisms on homology up to degree $\tfrac{n}{2} - 1$ and split-injections in all degrees.
\end{rcoro}

These are proved as corollaries of an extension of the main theorem of \cite{Palmer2018HomologicalstabilitymoduliI}. Let $G$ be a subgroup of $\mathrm{Diff}(P)$ and let $C_{nP}(M;G)$ be the space whose points consist of an unordered collection of pairwise disjoint submanifolds in the interior of $M$ that are each diffeomorphic to $P$ and parametrised modulo $G$, and such that the whole collection is isotopic to a standard collection of parallel copies of $P$ in a collar neighbourhood of $\partial M$. The main theorem of \cite{Palmer2018HomologicalstabilitymoduliI} is that this sequence is homologically stable as $n\to\infty$, as long as:

\begin{itemizeb}
\item[(a)] The dimensions $p$ and $m$ satisfy $p \leq \tfrac12(m-3)$.
\item[(b*)] The subgroup $G \leq \mathrm{Diff}(P)$ is open.
\end{itemizeb}

In this paper we extend this result to moduli spaces of \emph{labelled} disconnected submanifolds. Let $Z$ be a right $G$-space and $\pi \colon Z \to \mathrm{Emb}(P,M)$ a $G$-equivariant Serre fibration with path-connected fibres (plus some auxiliary data; see \S\ref{s:mod-labelled} for the precise definitions). Then $C_{nP}(M,Z;G)$ is the space of submanifolds of the interior of $M$ (as above) that are equipped with labels in the appropriate fibres of $Z/G$.

\begin{rthm}[\ref{tlabelled}] \textup{[page \pageref{tlabelled}]}
There are natural stabilisation maps $C_{nP}(M,Z;G) \to C_{(n+1)P}(M,Z;G)$ that induce split-injections on homology in all degrees, and -- under the hypotheses \textup{(a)} and \textup{(b*)} -- isomorphisms up to degree $\tfrac{n}{2} - 1$.
\end{rthm}

This, in turn, follows from the fact that the sequence $C_{nP}(M;G)$ of (unlabelled) moduli spaces is homologically stable also with respect to (polynomial) twisted coefficient systems:

\begin{rthm}[\ref{ttwisted}] \textup{[page \pageref{ttwisted}]}
The stabilisation maps $C_{nP}(M;G) \to C_{(n+1)P}(M;G)$ induce split-injections on homology twisted by any polynomial functor $T$. Under the hypotheses \textup{(a)} and \textup{(b*)}, if $T$ has finite degree $d$, they induce isomorphisms on homology twisted by $T$ up to degree $\tfrac{n-d}{2}$.
\end{rthm}

See \S\ref{s:ths} for the definition of ``polynomial functor'' (``finite-degree twisted coefficient system'') in this context. We prove this \emph{twisted} homological stability result as a consequence of the \emph{untwisted} homological stability result of \cite{Palmer2018HomologicalstabilitymoduliI}, adapting the techniques of \cite{Palmer2018Twistedhomologicalstability} to do so.

\phantomsection\label{diagram-implications}
In summary, the sequence of implications that we prove is:
\begin{center}
\begin{tikzpicture}
[x=1mm,y=1mm,font=\small]
\node (a) at (-5,0) {\cite{Palmer2018HomologicalstabilitymoduliI}};
\node (b) at (30,0) {Theorem \ref{ttwisted}};
\node (c) at (60,0) {Theorem \ref{tlabelled}};
\node (d) at (90,0) {Theorem \ref{tmain}};
\node (e) at (120,0) {Corollary \ref{coro-singularities}.};
\draw[very thick,->] (a) to node[above]{\S\ref{s:ths} and} node[below]{\cite{Palmer2018Twistedhomologicalstability}} (b);
\draw[very thick,->] (b) to node[above]{\S\ref{s:proof-labelled}} (c);
\draw[very thick,->] (c) to node[above]{\S\ref{s:proof-symm-diff-groups}} (d);
\draw[very thick,->] (d) to node[above]{\S\ref{s:conical-singularities}} (e);
\end{tikzpicture}
\end{center}

\begin{rmk}\label{r:after-linear-diagram}
In the special case where $P$ is a point, Theorem \ref{tmain} is Theorem 1.3 of \cite{Tillmann2016Homologystabilitysymmetric}, and states that the symmetric diffeomorphism groups of a sequence of manifolds obtained by iterated connected sum (in the usual sense) is homologically stable. When $P$ is a point, the manifolds $\mathbf{M}_n$ considered in Corollary \ref{coro-singularities} are in fact \emph{non}-singular, since in this case $\partial T = S^{m-1}$ and a singularity whose link (type) is a sphere is locally Euclidean. However, $\mathbf{M}_n$ is still a manifold with marked points, so in this case Corollary \ref{coro-singularities} is homological stability for diffeomorphism groups of manifolds with marked points, with respect to the number of marked points: this is Theorem 1.1 of \cite{Tillmann2016Homologystabilitysymmetric}.
\end{rmk}

\begin{rmk}
\label{rmk:surgery}
In the special case where $P \hookrightarrow N$ is the inclusion $S^p \hookrightarrow S^m$ (see Example \ref{eg:surgery}), Theorem \ref{tmain} is homological stability for the symmetric diffeomorphism groups of a sequence of manifolds obtained from $M$ by iterated $p$-surgery along a collection of pairwise disjoint, mutually isotopic, framed embeddings $S^p \times D^{m-p} \hookrightarrow M$, for $p \leq \tfrac12(m-3)$.
\end{rmk}

\subsection{The hypotheses}\label{ss:hypotheses}

\begin{rmk}
\label{rmk:Dehn-surgery}
The dimension condition (a) is not used in any of the proofs of the four implications in the diagram above; the only place where it is used is in the proof of homological stability for $C_{nP}(M;G)$ in \cite{Palmer2018HomologicalstabilitymoduliI}.\footnote{See Remark 1.8 of \cite{Palmer2018HomologicalstabilitymoduliI} for a discussion of exactly where the dimension hypothesis is used in that paper.} Thus we may apply the same implications to the main result of \cite{Kupers2020unlinkedcircles} instead of \cite{Palmer2018HomologicalstabilitymoduliI}, with the result that Theorems \ref{tmain}, \ref{tlabelled}, \ref{ttwisted} and Corollary \ref{coro-singularities} are also true if the dimension assumption (a) is replaced with the assumption that $P=S^1$, $M$ is a $3$-manifold and the embedding $P=S^1 \hookrightarrow \partial M$ extends to a $2$-disc.
\end{rmk}

Putting together Remarks \ref{rmk:surgery} and \ref{rmk:Dehn-surgery}, we have the following special cases of Theorem \ref{tmain}.

\begin{coro}
We have homological stability for the symmetric diffeomorphism groups of sequences of manifolds obtained from $M$ by:
\begin{itemizeb}
\item iterated surgery along a framed $p$-sphere that may be isotoped into $\partial M$, if $p \leq \tfrac12(\mathrm{dim}(M)-3)$,
\item iterated Dehn surgery along an unknot, if $\mathrm{dim}(M) = 3$.
\end{itemizeb}
\end{coro}

\begin{rmk}
The subgroup condition (b) is always satisfied if we take $H$ to be the full group $\mathrm{Diff}_{\mathrm{fib}}(T)$ of fibrewise diffeomorphisms of $T$. This is because the map $\mathrm{Diff}_{\mathrm{fib}}(T) \to \mathrm{Diff}(P)$, given by restriction along the zero-section, is a fibre bundle, so its image is open.
\end{rmk}

\begin{rmk}
When $P$ is a point, the fibrewise diffeomorphism group is the orthogonal group $O(m)$, and condition (b) says that $H$ must be a closed subgroup of $O(m)$, not contained in $SO(m)$. Condition (c) says that any element of $H$ must extend to a diffeomorphism of $N$. If $P$ is a point and $H \subseteq SO(m)$, then this is always possible. However, in view of condition (b), we know that $H$ must not be contained in $SO(m)$, and in this case condition (c) is satisfied if and only if $N$ is either non-orientable or admits an orientation-reversing diffeomorphism.
\end{rmk}

\subsection{Stable homology}
\label{ss:stable-homology}

Knowing that a sequence of groups or spaces is homologically stable motivates the question of identifying the \emph{stable homology} of the sequence, i.e.\ the colimit of the homology of the sequence. As far as the author is aware, this question is open, both for symmetric diffeomorphism groups (corresponding to Theorem \ref{tmain}) and for diffeomorphism groups of manifolds with conical singularities (corresponding to Corollary \ref{coro-singularities}). However, the stable homology in the latter case may be related to the work of Perlmutter~\cite{Perlmutter2015Cobordismcategorymanifolds,Perlmutter2013CobordismCategoryof} on cobordism categories of manifolds with Baas-Sullivan singularities.

For a discussion of the (also mostly open) question of identifying the stable homology of the moduli spaces of disconnected submanifolds $C_{nP}(M;G)$, see \S$\langle\text{vi.}\rangle$ of the introduction of \cite{Palmer2018HomologicalstabilitymoduliI}.

\begin{rmk}
\label{rmk:GriffinHatcher}
The author has been informed of forthcoming work of James Griffin and Allen Hatcher on the homology (both stable and unstable) of a space closely related to $C_{nS^1}(\bD^3)$. Here, we suppress mention of the subgroup $G \leq \mathrm{Diff}(S^1)$, which we take to be the full group of diffeomorphisms, so this is the moduli space of unoriented $n$-component unlinks in $\bR^3$. This has a subspace $\cR_n$ of all unlinks whose components are all \emph{round} (meaning a rotation, translation and dilation of the standard inclusion of $S^1$ into $\bR^3$), and it is shown in \cite{BrendleHatcher2013Configurationspacesrings} that the inclusion $\cR_n \hookrightarrow C_{nS^1}(\bD^3)$ is a homotopy equivalence. If we write $\cP\cR_n$ for the ordered version of this space, where the components of the unlink are numbered by $\{1,\ldots,n\}$, there is a projection $\cP\cR_n \to (\bR\bP^2)^n$ given by remembering the normal vectors of a configuration of round circles, which was shown in \cite{BrendleHatcher2013Configurationspacesrings} to be a quasifibration. The fibre is the space of ordered configurations of pairwise disjoint round circles in $\bR^3$ that are each contained in $\bR^2 \times \{h\}$ for some $h \in \bR$. The forthcoming result of Griffin and Hatcher is a computation of the integral homology of this fibre.
\end{rmk}

\subsection{Mapping class groups}
\label{ss:mcg}

In \S 5 of \cite{Tillmann2016Homologystabilitysymmetric} it is shown how to modify the methods of that paper --- whose main result is homological stability for symmetric diffeomorphism groups of manifolds with respect to connected sum at a point --- to prove homological stability instead for the symmetric \emph{mapping class groups}, in other words, the (discrete) groups of path-components of the symmetric diffeomorphism groups.

This depends on knowing homological stability for the fundamental groups $\pi_1(C_n(M))$ of configuration spaces (instead of homological stability for the configuration spaces themselves) as an input for the argument. The reason why $\pi_1(C_n(M))$ is homologically stable is:
\begin{itemizeb}
\item[(i)] If $M$ is a surface (which we are assuming to be connected and with non-empty boundary), then $C_n(M)$ is an aspherical space, so its homology is the same as the homology of its fundamental group.
\item[(ii)] If $\mathrm{dim}(M) \geq 3$, then $C_n(M)$ is not aspherical, but its fundamental group nevertheless has a configuration space model for its classifying space. Namely, $\pi_1(C_n(M))$ decomposes as the wreath product $\pi_1(M) \wr \Sigma_n$, as shown in Lemma 4.1 of \cite{Tillmann2016Homologystabilitysymmetric}, and a model for the classifying space of this wreath product is the labelled configuration space $C_n(\bR^\infty,B\pi_1(M))$, which is homologically stable.
\end{itemizeb}

When $P \neq \mathrm{pt}$, homological stability is in general \emph{not} known for the \emph{motion groups} $\pi_1(C_{nP}(M))$. The author is not aware of any $C_{nP}(M)$ that is aspherical, except when $P$ is a point and $M$ is a surface, so argument (i) does not help us. In particular, the moduli spaces $C_{nS^1}(\bD^3)$ are not aspherical. However, their fundamental groups nevertheless \emph{are} known to be homologically stable, by different means: they are isomorphic to certain quotients of mapping class groups of $3$-manifolds, and a special case of the main result of \cite{HatcherWahl2010Stabilizationmappingclass} implies that they are homologically stable. We may therefore adapt the methods of the present paper, as in \S 5 of \cite{Tillmann2016Homologystabilitysymmetric}, using as input the result of \cite{HatcherWahl2010Stabilizationmappingclass}, to deduce homological stability for the \emph{symmetric mapping class groups} of any sequence of $3$-manifolds obtained from $\bD^3$ by iterated parametric connected sum (with copies of a fixed manifold) along the components of an unlink.

It seems likely that argument (ii) above, i.e.\ the argument of \S 4 of \cite{Tillmann2016Homologystabilitysymmetric}, could be extended to obtain homological stability for the motion groups $\pi_1(C_{nP}(M))$ whenever the dimensions of $P$ and $M$ satisfy condition (a), i.e., $\mathrm{dim}(M) \geq 2.\mathrm{dim}(P) + 3$, and then to deduce homological stability for the corresponding symmetric mapping class groups with respect to iterated parametric connected sum.

\paragraph{Outline.}
\addcontentsline{toc}{subsection}{Outline}

In sections \ref{s:param-connected-sum} and \ref{s:symm-diff-groups} we give precise definitions of \emph{iterated parametric connected sum} and \emph{symmetric diffeomorphism groups}. Then in the short sections \ref{s:stab} and \ref{s:hs} we explain how to define stabilisation maps and state our first main result, Theorem \ref{tmain}. In section \ref{s:mod-labelled} we give a careful definition of the notion of \emph{moduli spaces of disconnected \textbf{labelled} submanifolds} that we will need, and state Theorem \ref{tlabelled}. In sections \ref{s:proof-symm-diff-groups}--\ref{s:proof-labelled} we prove our main results, as explained in the diagram of implications from \S\ref{ss:results}, which we repeat here for convenience:
\begin{center}
\begin{tikzpicture}
[x=1mm,y=1mm,font=\small]
\node (a) at (-5,0) {\cite{Palmer2018HomologicalstabilitymoduliI}};
\node (b) at (30,0) {Theorem \ref{ttwisted}};
\node (c) at (60,0) {Theorem \ref{tlabelled}};
\node (d) at (90,0) {Theorem \ref{tmain}};
\node (e) at (120,0) {Corollary \ref{coro-singularities},};
\draw[very thick,->] (a) to node[above]{\S\ref{s:ths} and} node[below]{\cite{Palmer2018Twistedhomologicalstability}} (b);
\draw[very thick,->] (b) to node[above]{\S\ref{s:proof-labelled}} (c);
\draw[very thick,->] (c) to node[above]{\S\ref{s:proof-symm-diff-groups}} (d);
\draw[very thick,->] (d) to node[above]{\S\ref{s:conical-singularities}} (e);
\end{tikzpicture}
\end{center}
where the longest and most delicate step is the deduction in \S\ref{s:proof-symm-diff-groups} of Theorem \ref{tmain} from Theorem \ref{tlabelled}.

\paragraph{Acknowledgements.}
\addcontentsline{toc}{subsection}{Acknowledgements}

The author would like to thank Federico Cantero Mor{\'a}n, S{\o}ren Galatius, James Griffin, Allen Hatcher, Geoffroy Horel, Alexander Kupers, Oscar Randal-Williams and Ulrike Tillmann for many enlightening discussions during the preparation of this article. He would also like to thank Diarmuid Crowley and Arkadiy Skopenkov for helpful discussions about parametric connected sum, and for pointing out several references.

\tableofcontents

\section{Iterated parametric connected sum}\label{s:param-connected-sum}

Let $M$ be a smooth connected $m$-dimensional manifold with boundary and write
\[
\hat{M} = M \underset{\partial M \times [0,\infty]}{\cup} (\partial M \times [-1,\infty]),
\]
where $\partial M \times [0,\infty]$ is considered as a subspace of $M$ via a choice of collar neighbourhood, namely a proper embedding $\mathrm{col} \colon \partial M \times [0,\infty] \hookrightarrow M$ sending $\partial M \times \{0\}$ to $\partial M$ by the obvious map. Also let $N$ be a smooth compact $m$-dimensional manifold and $P$ be a smooth closed $p$-dimensional manifold. Fix embeddings
\begin{align*}
i \colon P &\lhto \partial M \subseteq \hat{M}, \\
j \colon P &\lhto \mathring{N} = \mathrm{int}(N),
\end{align*}
assume that the normal bundles $\nu_i \to P$ and $\nu_j \to P$ are isomorphic, and choose an isomorphism between them. Choose a metric on the bundle $\nu_i$ so that its structure group is $O(m-p)$, and give $\nu_j$ the corresponding metric via the chosen isomorphism. Let $D(\nu_i) \subseteq \nu_i$ and $D(\nu_j) \subseteq \nu_j$ denote the subbundles consisting of vectors of norm at most one. We implicitly identify $D(\nu_i)$ and $D(\nu_j)$ via the chosen isomorphism, and write
\[
\xi \colon T = D(\nu_i) = D(\nu_j) \longrightarrow P,
\]
which is a fibre bundle with fibres diffeomorphic to the closed disc $D^{m-p}$ and structure group $O(m-p)$. Write $o \colon P \hookrightarrow T$ for the zero-section. We now also choose tubular neighbourhoods for $\nu_i$ and $\nu_j$, namely embeddings
\begin{align*}
\tau_i \colon T &\lhto \hat{M}, \\
\tau_j \colon T &\lhto \mathring{N},
\end{align*}
such that $i = \tau_i \circ o$ and $j = \tau_j \circ o$, and assume that $\tau_i(T) \subseteq \partial M \times (-\tfrac12,\tfrac12) \subset \hat{M}$. We may define an embedding
\[
\Phi \colon nT = \{ 1,\ldots, n \} \times T \lhto \mathring{M}
\]
of $n$ disjoint, parallel copies of the tubular neighbourhood $T$ in the interior of $M$ by
\[
\Phi(\alpha,x) = \tau_i(x) + \alpha - \tfrac12,
\]
where for $r \in [0,\infty)$ the notation $\! {} + r$ denotes the self-map $\partial M \times [-1,\infty] \to \partial M \times [-1,\infty]$ given by the identity on $\partial M$ and adding $r$ in the second coordinate.

\begin{notation}
For a space $X$, we will henceforth use the notations $nX$ and $\{1,\ldots,n\} \times X$ interchangeably. We will also write $\hat{n} X$ for $\{ 0,\ldots,n \} \times X$. Note that $\hat{n}X \supsetneq nX$!
\end{notation}

\begin{defn}\label{d:param-conn-sum}
With this data, we may form the \emph{parametric connected sum}
\[
M \underset{nP}{\sharp} nN \;=\; (M \smallsetminus \Phi(nT')) \; \underset{n(\mathring{T} \smallsetminus T')}{\cup} \; n(N \smallsetminus \tau_j(T')),
\]
where $\mathring{T} \to P$ is the interior of $T$, equivalently the subbundle of vectors of length less than $1$, and $T' \to P$ is the subbundle of $T$ of vectors of length at most one half. The union is formed along $n(\mathring{T} \smallsetminus T')$, which is viewed as a subspace of $M \smallsetminus \Phi(nT')$ via $\Phi$, and as a subspace of $n(N \smallsetminus \tau_j(T'))$ via $\mathrm{id} \times \tau_j$ precomposed with the involution of $\mathring{T} \smallsetminus T' = \partial T \times (0.5,1)$ given by $(x,t) \mapsto (x,1.5-t)$.
\end{defn}

\begin{rmk}
The boundary of the parametric connected sum is $\partial \bigl( M \underset{nP}{\sharp} nN \, \bigr) \,\cong\, \partial M \sqcup n(\partial N)$.
\end{rmk}

\section{Symmetric diffeomorphism groups}\label{s:symm-diff-groups}

Recall that we have a disc bundle $\xi \colon T \to P$ with structure group $O(m-p)$. Let
\[
\mathrm{Diff}_{\mathrm{fib}}(T) \leq \mathrm{Diff}(T)
\]
denote the subgroup of diffeomorphisms $\varphi$ such that $\xi \circ \varphi = \bar{\varphi} \circ \xi$ for some diffeomorphism $\bar{\varphi}$ of $P$ and $\varphi$ restricts to a linear isometry on each fibre of $\xi$. Write 
\[
z \colon \mathrm{Diff}_{\mathrm{fib}}(T) \longrightarrow \mathrm{Diff}(P)
\]
for the continuous homomorphism given by $\varphi \mapsto \bar{\varphi}$, equivalently, by restricting fibrewise diffeomorphisms of $T$ to $P$ via the zero-section of $\xi$. Its kernel is the group of bundle automorphisms of $\xi$. Choose a subgroup
\[
H \leq \mathrm{Diff}_{\mathrm{fib}}(T),
\]
and write $G = z(H) \leq \mathrm{Diff}(P)$ and $K = H \cap \mathrm{ker}(z) \leq \mathrm{ker}(z)$.

\begin{defn}\label{d:symm-diff-group}
The \emph{symmetric diffeomorphism group} $\Sigma_H \mathrm{Diff}(M \underset{nP}{\sharp} nN) \leq \mathrm{Diff}(M \underset{nP}{\sharp} nN)$ consists of those diffeomorphisms that are the identity on a neighbourhood of $\partial M$, send the submanifold
\[
n(\mathring{T} \smallsetminus T') \;\subset\; M \underset{nP}{\sharp} nN
\]
to itself setwise and act on this submanifold through the wreath product $H \wr \Sigma_n$.

The subgroup $\Sigma_H \mathrm{Diff}(M,nT)$ of $\mathrm{Diff}(M)$ consists of those diffeomorphisms that are the identity on a neighbourhood of the boundary, send the submanifold
\[
\Phi(nT) \;\subset\; M
\]
to itself setwise and act on this submanifold through the wreath product $H \wr \Sigma_n$.
\end{defn}

\begin{rmk}\label{r:mphi}
Each diffeomorphism of $\Sigma_H \mathrm{Diff}(M,nT)$ is determined by its restriction to the submanifold $M_\Phi = M \smallsetminus \Phi(n\mathring{T})$, whose boundary splits as $\partial M_\Phi = \partial M \sqcup \Phi(n\partial T)$, so it may also be viewed as the subgroup of $\mathrm{Diff}(M_\Phi)$ of diffeomorphisms that act by the identity on a neighbourhood of $\partial M \subset \partial M_\Phi$ and by $H \wr \Sigma_n$ on $\Phi(n\partial T) \subset \partial M_\Phi$. We will call this the \emph{boundary-permuting diffeomorphism group} of $M_\Phi$. See \S\ref{s:conical-singularities} for another interpretation in terms of manifolds with conical singularities.
\end{rmk}

\begin{rmk}\label{r:pullback-of-groups}
There is a continuous homomorphism
\begin{equation}\label{eq:comparision-hom}
\Sigma_H \mathrm{Diff}(M \underset{nP}{\sharp} nN) \longrightarrow \Sigma_H \mathrm{Diff}(M,nT)
\end{equation}
given by restricting a diffeomorphism to $M \smallsetminus \Phi(nT')$ and then extending to $M$ by extending linearly across each fibre of $\xi \colon T \to P$. Similarly, there is a continuous homomorphism
\[
\Sigma_H \mathrm{Diff}(M \underset{nP}{\sharp} nN) \longrightarrow \mathrm{Diff}_H(N) \wr \Sigma_n,
\]
where $\mathrm{Diff}_H(N) \leq \mathrm{Diff}(N)$ denotes the subgroup of diffeomorphisms that send the submanifold $\tau_j(T) \subset N$ to itself setwise and act on it through the subgroup $H \leq \mathrm{Diff}(T)$. These homomorphisms fit into the following pullback square of topological groups:
\begin{equation}\label{eq:pullback-symm-boundary}
\centering
\begin{split}
\begin{tikzpicture}
[x=1mm,y=1mm]
\node (tl) at (0,15) {$\Sigma_H \mathrm{Diff}(M \underset{nP}{\sharp} nN)$};
\node (tr) at (50,15) {$\mathrm{Diff}_H(N) \wr \Sigma_n$};
\node (bl) at (0,0) {$\Sigma_H \mathrm{Diff}(M,nT)$};
\node (br) at (50,0) {$H \wr \Sigma_n$};
\draw[->] ($ (tl.east) + (0,0.75) $) to ($ (tr.west) + (0,0.75) $);
\draw[->] (bl) to (br);
\draw[->] ($ (tl.south) + (0,2) $) to (bl);
\draw[->] (tr) to (br);
\end{tikzpicture}
\end{split}
\end{equation}
Via this description, we could generalise the notion of \emph{symmetric diffeomorphism group}, exactly as on page 133 of \cite{Tillmann2016Homologystabilitysymmetric} (the diagram \eqref{eq:pullback-symm-boundary} corresponds exactly to the diagram at the top of that page), by replacing $\mathrm{Diff}_H(N)$ with an arbitrary topological group $L$ equipped with a surjective\footnote{We will shortly impose the assumption that $\mathrm{Diff}_H(N) \to H$ is surjective, in other words, that each element of $H \leq \mathrm{Diff}(T)$ may be extended over $N$. We do not, however, need this extension to preserve composition, i.e.\ we do not require this surjection to split.} continuous homomorphism $L \to H$, and then taking the pullback of the diagram
\[
\Sigma_H \mathrm{Diff}(M,nT) \longrightarrow H \wr \Sigma_n \longleftarrow L \wr \Sigma_n.
\]
The results of this paper hold also in this higher level of generality, but for concreteness we will stick to the symmetric diffeomorphism groups as defined above, with $L = \mathrm{Diff}_H(N)$ for some manifold $N$ equipped with an embedding $T \hookrightarrow \mathring{N}$.
\end{rmk}

\begin{rmk}\label{r:canonical-heq}
Up to a canonical homotopy equivalence, the boundary-permuting diffeomorphism group $\Sigma_H \mathrm{Diff}(M,nT)$ is a special case of the symmetric diffeomorphism group $\Sigma_H \mathrm{Diff}(M \underset{nP}{\sharp} nN)$.

To see this, let us fix the initial data of an embedding $i \colon P \hookrightarrow \hat{M}$ whose image lies in $\partial M \subset \hat{M}$, a metric on the normal bundle $\nu_i \to P$ of this embedding and a tubular neighbourhood $T = D(\nu_i) \hookrightarrow \hat{M}$. Together with a choice of subgroup $H \leq \mathrm{Diff}_{\mathrm{fib}}(T)$, this determines the group $\Sigma_H \mathrm{Diff}(M,nT)$.

We are now free to choose any compact $m$-dimensional manifold $N$, an embedding $j \colon P \hookrightarrow \mathring{N}$ and isomorphism $\nu_i \cong \nu_j$. In particular, we may choose $N = D_2(\nu_i)$, the total space of the subbundle of $\nu_i$ consisting of all vectors of norm at most two, and let $j \colon P \hookrightarrow \mathring{N} = \mathring{D}_2(\nu_i)$ be the zero-section. The normal bundles $\nu_i$ and $\nu_j$ are then canonically isomorphic. We also have to choose a tubular neighbourhood
\[
T = D(\nu_j) \lhto \mathring{D}_2(\nu_j) = \mathring{N},
\]
which we simply take to be the inclusion. It is then easy to see that, in this case, the parametric connected sum $M \underset{nP}{\sharp} nD_2(\nu_i)$ deformation retracts onto its submanifold $M_\Phi$ (\cf Remark \ref{r:mphi}). Moreover, in this case, the continuous homomorphism \eqref{eq:comparision-hom} admits a continuous, homomorphic section given by extending fibrewise automorphisms of $D(\nu_i) \to P$ linearly to $D_2(\nu_i) \to P$, and this is a homotopy inverse (in the $2$-category of topological groups) for \eqref{eq:comparision-hom}.\footnote{Here it is important that, in Definition \ref{d:symm-diff-group}, we require symmetric diffeomorphisms to act as the identity near $\partial M$, but there is no condition on how they act near the rest of the boundary of $M \underset{nP}{\sharp} nN$, i.e.\ the $n$ copies of $\partial N$.}
\end{rmk}

\section{Stabilisation maps}\label{s:stab}

We may extend $\Phi$ to an embedding
\[
\hat{\Phi} \colon \hat{n}T = \{ 0,\ldots,n \} \times T \lhto \hat{M}
\]
defined by the same formula $\hat{\Phi}(\alpha,x) = \tau_i(x) + \alpha - \tfrac12$ as before. In other words, we adjoin the embedding $\tau_i - \tfrac12$ to $\Phi$. Using this $\hat{\Phi}$ and $n+1$ copies of the tubular neighbourhood $\tau_j \colon T \hookrightarrow N$ we define
\[
M \underset{\hat{n}P}{\sharp} \hat{n}N \;=\; (\hat{M} \smallsetminus \hat{\Phi}(\hat{n}T')) \; \underset{\hat{n}(\mathring{T} \smallsetminus T')}{\cup} \; \hat{n}(N \smallsetminus \tau_j(T'))
\]
as above, as well as groups $\Sigma_H \mathrm{Diff}(M \underset{\hat{n}P}{\sharp} \hat{n}N)$ and $\Sigma_H \mathrm{Diff}(M,\hat{n}T)$ and a continuous homomorphism
\[
\Sigma_H \mathrm{Diff}(M \underset{\hat{n}P}{\sharp} \hat{n}N) \longrightarrow \Sigma_H \mathrm{Diff}(M,\hat{n}T).
\]
Extending diffeomorphisms by the identity along the inclusion
$M \underset{nP}{\sharp} nN \lhto M \underset{\hat{n}P}{\sharp} \hat{n}N$, we obtain the horizontal maps in the commutative square
\begin{equation}\label{eq:symdiff-square}
\centering
\begin{split}
\begin{tikzpicture}
[x=1mm,y=1mm]
\node (tl) at (0,15) {$\Sigma_H \mathrm{Diff}(M \underset{nP}{\sharp} nN)$};
\node (tr) at (50,15) {$\Sigma_H \mathrm{Diff}(M \underset{\hat{n}P}{\sharp} \hat{n}N)$};
\node (bl) at (0,0) {$\Sigma_H \mathrm{Diff}(M,nT)$};
\node (br) at (50,0) {$\Sigma_H \mathrm{Diff}(\hat{M},\hat{n}T)$};
\draw[->] ($ (tl.east) + (0,1.5) $) to ($ (tr.west) + (0,1.5) $);
\draw[->] (bl) to (br);
\draw[->] ($ (tl.south) + (0,2) $) to (bl);
\draw[->] ($ (tr.south) + (0,2) $) to (br);
\end{tikzpicture}
\end{split}
\end{equation}
of topological groups. These are the \emph{stabilisation maps}.

\section{Homological stability for symmetric diffeomorphism groups}\label{s:hs}

We now make three assumptions (\cf \S\ref{ss:hypotheses}).
\begin{itemizeb}
\item[(a)] The dimensions $p$ and $m$ satisfy $p \leq \tfrac12(m-3)$.
\item[(b)] The subgroup $G = z(H) \leq \mathrm{Diff}(P)$ is open and $K \leq \mathrm{ker}(z)$ is closed. Moreover, we assume that the coset space $\mathrm{ker}(z)/K$ is path-connected.
\item[(c)] Every diffeomorphism in $H$ extends across $N$. More precisely, given any $h \in H \leq \mathrm{Diff}(T)$, there exists a diffeomorphism $\varphi$ of $N$ such that $\varphi(\tau_j(T)) = \tau_j(T)$ and $\varphi|_{\tau_j(T)} = \tau_j \circ h \circ \tau_j^{-1}$. In the notation of Remark \ref{r:pullback-of-groups}, this says that the continuous homomorphism $\mathrm{Diff}_H(N) \to H$ is surjective (but not necessarily split).
\end{itemizeb}

The first main result of this paper is the following theorem.

\begin{athm}\label{tmain}
Under these assumptions, the two horizontal morphisms in \eqref{eq:symdiff-square} induce split-injections on homology in all degrees and isomorphisms in degrees $* \leq \tfrac{n}{2} - 1$.
\end{athm}

\begin{rmk}
If we take coefficients in a field or $\bQ/\bZ$, the range of surjectivity improves to $* \leq \tfrac{n}{2}$.
\end{rmk}

\section{Moduli spaces of labelled disconnected submanifolds}\label{s:mod-labelled}

We will deduce Theorem \ref{tmain} from a generalisation of the main theorem of \cite{Palmer2018HomologicalstabilitymoduliI}, so we will first state this generalisation precisely (see Theorem \ref{tlabelled} on page \pageref{tlabelled}).

\begin{notation}\label{not:Mr}
By construction of $\hat{M}$, there is a smooth embedding
\[
\colhat \colon \partial M \times [-1,\infty] \lhto \hat{M}.
\]
For any $r \in [-1,\infty)$ we will write $M(r) = \hat{M} \smallsetminus \colhat([-1,r])$ and $M[r] = \hat{M} \smallsetminus \colhat([-1,r))$.
\end{notation}

\begin{defn}
Write $\mathrm{Diff}_{[-1,1]}(\bR)$ for the topological group of diffeomorphisms $\bR \to \bR$ whose support is contained in $[-1,1] \subset \bR$. There is an evaluation map
\[
\mathrm{ev}_0 \colon \mathrm{Diff}_{[-1,1]}(\bR) \longrightarrow (-1,1)
\]
taking $\varphi$ to $\varphi(0)$, which is a fibre bundle. We now make some choices that will be used in several constructions in this section and in subsequent sections. First, we choose an odd\footnote{In the sense of odd functions, i.e.\ $\theta(-t) = -\theta(t)$.} diffeomorphism $\theta \colon \bR \to (-1,1)$. Second, we choose a lift $\bar{\theta} \colon \bR \to \mathrm{Diff}_{[-1,1]}(\bR)$ of $\theta$ (i.e.\ $\mathrm{ev}_0 \circ \bar{\theta} = \theta$) that is also a homomorphism with respect to addition on $\bR$ and composition of diffeomorphisms.
\end{defn}

\begin{rmk}
Given $\theta$, one may try to define $\bar{\theta}(r) \colon \bR \to \bR$ on $t \in (-1,1)$ by
\[
\bar{\theta}(r)(t) = \theta(\theta^{-1}(t) + r)
\]
and extend by the identity outside of $(-1,1)$. This will work as long as the function so defined is smooth at $t=1$, which depends on how well $\theta$ has been chosen. For example, if we only care about $C^1$ diffeomorphisms, the function $\theta(t) = \tfrac{2}{\pi}\mathrm{arctan}(t)$ would work for this construction of $\bar{\theta}$.
\end{rmk}

\begin{defn}\label{d:shift}
Recall that we chose an embedding $i \colon P \hookrightarrow \partial M$. For any $r \in \bR$ we denote the shifted embedding
\[
\colhat \circ (i(-),\theta(r)) \colon P \lhto \hat{M}
\]
by $i_r$ and we define a diffeomorphism
\[
\mathrm{sh}_r \colon \hat{M} \longrightarrow \hat{M}
\]
by $\mathrm{sh}_r(\colhat(x,t)) = \colhat(x,\bar{\theta}(r)(t))$ and $\mathrm{sh}_r(x) = x$ if $x \in \hat{M}$ is not in the image of $\colhat$.

Now write $E = \mathrm{Emb}(P,\hat{M})$. There is a continuous group homomorphism $\gamma \colon \bR \to \mathrm{Homeo}(E)$ given by $\gamma(r)(\varphi) = \mathrm{sh}_r \circ \varphi$. Moreover, since $G \leq \mathrm{Diff}(P)$ acts on $E$ by precomposition, and $\gamma(r)$ acts by postcomposition, we in fact have a continuous group homomorphism
\[
\gamma \colon \bR \longrightarrow \mathrm{Homeo}^G(E)
\]
into the topological group of $G$-equivariant self-homeomorphisms of $E$. Note that $\gamma(r)(i_s) = i_{r+s}$ for all $r,s \in \bR$.
\end{defn}

\begin{inputdata}\label{input-data}
Now we fix the additional input data needed to define moduli spaces of disconnected submanifolds \emph{with labels} (see Definition \ref{d:labelled-disconn-submfld}). Choose a $G$-equivariant Serre fibration $\pi \colon Z \to E$ and a continuous homomorphism
\[
\bar{\gamma} \colon \bR \longrightarrow \mathrm{Homeo}^G(Z)
\]
such that $\pi \circ \bar{\gamma}(r) = \gamma(r) \circ \pi$ for all $r \in \bR$. Also choose a basepoint $\bar{\imath}_0 \in Z$ such that $\pi(\bar{\imath}_0) = i_0$. For any $r \in \bR$, define $\bar{\imath}_r = \bar{\gamma}(r)(\bar{\imath}_0)$ and note that $\pi(\bar{\imath}_r) = i_r$.

(The purpose of the data $(\bar{\gamma},\bar{\imath}_0)$ is to allow us to lift the operation of ``shifting'' an embedding via $\mathrm{sh}_r \circ -$ from the space of embeddings to the total space $Z$ of the fibration.)
\end{inputdata}

\begin{defn}[\emph{Moduli spaces of labelled disconnected submanifolds}]\label{d:labelled-disconn-submfld}
Fix a subgroup $G \leq \mathrm{Diff}(P)$ as above. Then $\pi^n \colon Z^n \to E^n$ is a $(G \wr \Sigma_n)$-equivariant fibration. Let $\pi_n \colon Z_n \to \mathrm{Emb}(nP,M(0))$ be its restriction to the $(G \wr \Sigma_n)$-invariant subspace $\mathrm{Emb}(nP,M(0)) \subset E^n$. We then define
\[
C_{nP}(M,Z;G) \;=\; \bigl( Z_n / (G \wr \Sigma_n) \bigr)_{\{ [\bar{\imath}_1] ,\ldots, [\bar{\imath}_n] \}} ,
\]
the path-component of the quotient $Z_n / (G \wr \Sigma_n) \subseteq \mathrm{Sp}^n(Z/G)$ containing the element $\{ [\bar{\imath}_1] ,\ldots, [\bar{\imath}_n] \}$. Similarly, we may restrict $\pi^{n+1}$ to the subspace $\mathrm{Emb}(\hat{n}P,M(-1)) \subset E^{n+1}$ to obtain a $(G \wr \Sigma_{n+1})$-equivariant fibration $\pi_{\hat{n}} \colon Z_{\hat{n}} \to \mathrm{Emb}(\hat{n}P,M(-1))$, and define
\[
C_{\hat{n}P}(M,Z;G) \;=\; \bigl( Z_{\hat{n}} / (G \wr \Sigma_{n+1}) \bigr)_{\{ [\bar{\imath}_0] ,\ldots, [\bar{\imath}_n] \}} .
\]
Viewing these as subspaces of the symmetric powers $\mathrm{Sp}^n(Z/G)$ and $\mathrm{Sp}^{n+1}(Z/G)$ respectively, we may define a map
\[
s_n \colon C_{nP}(M,Z;G) \longrightarrow C_{\hat{n}P}(M,Z;G)
\]
by $s_n(\{ [\varphi_1] ,\ldots, [\varphi_n] \}) = \{ [\bar{\imath}_0] , [\varphi_1] ,\ldots, [\varphi_n] \}$.

These constructions are functorial in $P$, $M$, $G$ and $Z$ in an appropriate sense. We will describe how they are functorial in $Z$ when the other data $P$, $M$ and $G$ are fixed. For $i=1,2$ let $\pi_i \colon Z_i \to E$ be based, $G$-equivariant Serre fibrations and let $\bar{\gamma}_i \colon \bR \to \mathrm{Homeo}^G(Z_i)$ be continuous homomorphisms such that $\pi_i \circ \bar{\gamma}_i(r) = \gamma(r) \circ \pi_i$ for all $r$. Given any based, $G$-equivariant map $F \colon Z_1 \to Z_2$ such that $\pi_2 \circ F = \pi_1$ and $F \circ \bar{\gamma}_1(r) = \bar{\gamma}_2(r) \circ F$ for all $r$, there are induced maps making the square
\begin{center}
\begin{tikzpicture}
[x=1mm,y=1mm]
\node (tl) at (0,12) {$C_{nP}(M,Z_1;G)$};
\node (tr) at (40,12) {$C_{\hat{n}P}(M,Z_1;G)$};
\node (bl) at (0,0) {$C_{nP}(M,Z_2;G)$};
\node (br) at (40,0) {$C_{\hat{n}P}(M,Z_2;G)$};
\draw[->] (tl) to node[above,font=\small]{$s_n$} (tr);
\draw[->] (bl) to node[above,font=\small]{$s_n$} (br);
\draw[->] (tl) to (bl);
\draw[->] (tr) to (br);
\end{tikzpicture}
\end{center}
commute. These form a category with terminal object given by $(\pi,\bar{\gamma}) = (\mathrm{id} \colon E \to E,\gamma)$. When the fibration $\pi \colon Z \to E$ (and the map $\bar{\gamma}$) are taken to be this terminal object, we drop the $Z$ from the notation and write simply $s_n \colon C_{nP}(M;G) \to C_{\hat{n}P}(M;G)$. For any other choice of $(\pi \colon Z \to E,\bar{\gamma})$ there is a commutative square
\begin{equation}\label{eq:square-stabilisation}
\centering
\begin{split}
\begin{tikzpicture}
[x=1mm,y=1mm]
\node (tl) at (0,12) {$C_{nP}(M,Z;G)$};
\node (tr) at (40,12) {$C_{\hat{n}P}(M,Z;G)$};
\node (bl) at (0,0) {$C_{nP}(M;G)$};
\node (br) at (40,0) {$C_{\hat{n}P}(M;G).$};
\draw[->] (tl) to node[above,font=\small]{$s_n$} (tr);
\draw[->] (bl) to node[above,font=\small]{$s_n$} (br);
\draw[->] (tl) to (bl);
\draw[->] (tr) to (br);
\end{tikzpicture}
\end{split}
\end{equation}
\end{defn}

\begin{thm}[{\cite[Theorem A]{Palmer2018HomologicalstabilitymoduliI}}]
The map $s_n \colon C_{nP}(M;G) \to C_{\hat{n}P}(M;G)$ induces split-injections on homology in all degrees. It induces isomorphisms on homology up to degree $\tfrac{n}{2}$ if $p \leq \tfrac12(m-3)$ and $G$ is an open subgroup of $\mathrm{Diff}(P)$.
\end{thm}

Recall that $p$ and $m$ are the dimensions of $P$ and $M$ respectively. We will lift this to the top horizontal map in \eqref{eq:square-stabilisation}, under a condition on the fibration $\pi \colon Z \to E$.

\begin{athm}\label{tlabelled}
The map $s_n \colon C_{nP}(M,Z;G) \to C_{\hat{n}P}(M,Z;G)$ induces split-injections on homology in all degrees. If the fibres of $\pi$ are path-connected, $p \leq \tfrac12(m-3)$ and $G$ is an open subgroup of $\mathrm{Diff}(P)$, then it induces isomorphisms on integral homology up to degree $\tfrac{n}{2} - 1$ and on homology with coefficients in a field up to degree~$\tfrac{n}{2}$.
\end{athm}

\begin{rmk}
If a map induces isomorphisms on homology (up to a certain degree) with coefficients in any field, then it also induces isomorphisms up to the same degree with coefficients in any ring that may be constructed from fields by iterated (ring) extensions and colimits. In particular, $\bQ/\bZ$ is such a ring, so the conclusion of the above theorem also implies that $s_n$ induces isomorphisms on homology with $\bQ/\bZ$ coefficients up to degree $\tfrac{n}{2}$.
\end{rmk}

\section{Proof of stability for symmetric diffeomorphism groups}\label{s:proof-symm-diff-groups}

We will deduce Theorem \ref{tmain} from Theorem \ref{tlabelled} by a spectral sequence comparison argument. First we need some more constructions to set up the appropriate map of spectral sequences.

\subsection{Some fibre bundles}

\begin{notation}
For any real number $r$ we will write $\bR^{\infty,r} = \bR^\infty \times [r,\infty)$. Most of the time $r$ will be $0$, $-1$ or $-2$.
\end{notation}

Fix an embedding $b \colon \partial M \hookrightarrow \bR^\infty$. Write $\mathrm{Emb}_b(M,\rinftyz)$ for the space of embeddings $e \colon M \hookrightarrow \rinftyz$ such that
\begin{itemizeb}
\item[(i)] $e \circ \mathrm{\colhat}|_{\partial M \times [0,\epsilon)} = b \times \mathrm{incl}$ for some $\epsilon > 0$.
\end{itemizeb}
We then define
\[
X_n \subseteq \mathrm{Emb}_b(M,\rinftyz) \times \mathrm{Emb}(M \underset{nP}{\sharp} nN,\rinftyz)
\]
to be the subspace of pairs of embeddings $(e,f)$ such that
\begin{itemizeb}
\item[(ii)] $e|_{M \smallsetminus \Phi(nT')} = f|_{M \smallsetminus \Phi(nT')}$.
\end{itemizeb}
Note that there is a continuous action of the symmetric diffeomorphism group $\Sigma_H \mathrm{Diff}(M \underset{nP}{\sharp} nN)$ on $X_n$ by precomposition in each factor (and the homomorphism \eqref{eq:comparision-hom} for the first factor). Similarly, we write $\mathrm{Emb}_b(\hat{M},\rinftym)$ for the space of embeddings $e \colon \hat{M} \hookrightarrow \rinftym$ such that
\begin{itemizeb}
\item[(\^{i})] $e \circ \mathrm{\colhat}|_{\partial M \times [-1,-1+\epsilon)} = b \times \mathrm{incl}$ for some $\epsilon > 0$
\end{itemizeb}
and we define
\[
X_{\hat{n}} \subseteq \mathrm{Emb}_b(\hat{M},\rinftym) \times \mathrm{Emb}(M \underset{\hat{n}P}{\sharp} \hat{n}N,\rinftym)
\]
to be the subspace of pairs of embeddings $(e,f)$ such that
\begin{itemizeb}
\item[(\^{i}\^{i})] $e|_{\hat{M} \smallsetminus \hat{\Phi}(\hat{n}T')} = f|_{\hat{M} \smallsetminus \hat{\Phi}(\hat{n}T')}$.
\end{itemizeb}
There is a continuous action of the symmetric diffeomorphism group $\Sigma_H \mathrm{Diff}(M \underset{\hat{n}P}{\sharp} \hat{n}N)$ on $X_{\hat{n}}$, given by precomposition in each factor and the right-hand vertical map of \eqref{eq:symdiff-square} for the first factor.

\begin{lem}\label{l:pbundle}
The spaces $X_n$ and $X_{\hat{n}}$ are contractible and the quotient maps
\[
X_n \longrightarrow X_n / \Sigma_H \mathrm{Diff}(M \underset{nP}{\sharp} nN) \qquad\qquad X_{\hat{n}} \longrightarrow X_{\hat{n}} / \Sigma_H \mathrm{Diff}(M \underset{\hat{n}P}{\sharp} \hat{n}N)
\]
are principal bundles with structure groups $\Sigma_H \mathrm{Diff}(M \underset{nP}{\sharp} nN)$ and $\Sigma_H \mathrm{Diff}(M \underset{\hat{n}P}{\sharp} \hat{n}N)$ respectively.
\end{lem}
\begin{proof}
The contractibility of the spaces $X_n$ and $X_{\hat{n}}$ may be seen by the usual argument for the contractibility of spaces of embeddings into $\bR^\infty$: any family of such embeddings parametrised by $S^i$ is contained in some finite-dimensional subspace of $\bR^\infty$ (by compactness of $S^i$), and this may be used to extend it to $D^{i+1}$.

A mild extension of Proposition 4.15 of \cite{Palmer2018HomologicalstabilitymoduliI} shows that the action of $\mathrm{Diff}_\partial(\rinftyz)$ on the quotient $X_n / \Sigma_H \mathrm{Diff}(M \underset{nP}{\sharp} nN)$ is locally retractile. Then Proposition 4.8 of \cite{Palmer2018HomologicalstabilitymoduliI} implies that the projection of $X_n$ onto this quotient is a principal bundle. An identical argument implies that the other projection is also a principal bundle.
\end{proof}

This gives us explicit models for the classifying spaces of $\Sigma_H \mathrm{Diff}(M \underset{nP}{\sharp} nN)$ and $\Sigma_H \mathrm{Diff}(M \underset{\hat{n}P}{\sharp} \hat{n}N)$.

Write $\mathrm{Diff}_\partial(M)$ for the group of diffeomorphisms of $M$ that act by the identity on a neighbourhood of its boundary. There is a forgetful map
\begin{equation}\label{eq:main-fibre-bundle}
\Psi \colon X_n / \Sigma_H \mathrm{Diff}(M \underset{nP}{\sharp} nN) \longrightarrow \mathrm{Emb}_b(M,\rinftyz) / \mathrm{Diff}_\partial(M),
\end{equation}
and an analogous map $\hat{\Psi}$, replacing $n$ with $\hat{n}$ and $\rinftyz$ with $\rinftym$.

\begin{lem}\label{l:bundlepsi}
The maps $\Psi$ and $\hat{\Psi}$ are fibre bundles.
\end{lem}
\begin{proof}
As above, by a mild extension of Proposition 4.15 of \cite{Palmer2018HomologicalstabilitymoduliI}, the action of $\mathrm{Diff}_\partial(\rinftyz)$ on the quotient $\mathrm{Emb}_b(M,\rinftyz)/\mathrm{Diff}_\partial(M)$ is locally retractile, and then Theorem A of \cite{Palais1960Localtrivialityof} (see also Proposition 4.7 of \cite{Palmer2018HomologicalstabilitymoduliI}) implies that $\Psi$ is a fibre bundle (and an identical argument implies the same for $\hat{\Psi}$).
\end{proof}

\subsection{Moduli spaces of submanifolds labelled by parametric-connected-sum-data}

Recall from \S\ref{s:symm-diff-groups} that $K = H \cap \mathrm{ker}(z) \leq H \leq \mathrm{Diff}_{\mathrm{fib}}(T)$, and that $G = z(H) \leq \mathrm{Diff}(P)$, where
\[
z \colon \mathrm{Diff}_{\mathrm{fib}}(T) \longrightarrow \mathrm{Diff}(P)
\]
is the restriction along the zero-section $o \colon P \hookrightarrow T$ of $\xi \colon T \to P$.

\begin{notation}\label{not:nprime}
Write $N' = N \smallsetminus \tau_j(T')$ and $U = \mathring{T} \smallsetminus T'$. There is an involution
\[
\sigma \colon U \longrightarrow U
\]
given by $(x,t) \mapsto (x,1.5-t)$, where we identify $U$ with $\partial T \times (0.5,1)$ (\cf Definition \ref{d:param-conn-sum}). Note also that $\tau_j$ restricts to an embedding $U \hookrightarrow N'$.
\end{notation}

\begin{defn}
The subgroup $\mathrm{Diff}_K(N')$ of $\mathrm{Diff}(N')$ consists of those diffeomorphisms that send the subset $\tau_j(U) \subset N'$ to itself and act on it via $\tau_j K \tau_j^{-1}$.
\end{defn}

\begin{construction}\label{construction:Z}
Fix an embedding $e_0 \in \mathrm{Emb}_b(M,\rinftyz)$. For convenience, we assume that
\begin{itemizeb}
\item[(a)] $e_0(\mathrm{col}(x,t)) = (b(x),t)$ for $(t,x) \in \partial M \times [0,1]$,
\item[(b)] $e_0(M \smallsetminus \mathrm{col}(\partial M \times [0,1])) \subseteq \bR^\infty \times (1,\infty)$.
\end{itemizeb}
The fact that we are making this assumption will not cause problems later, since $\mathrm{Emb}_b(M,\rinftyz)$ is path-connected (in fact contractible). Also, it will be convenient to extend $\hat{M}$ slightly further to
\[
\doublehat{M} = M \underset{\partial M \times [0,\infty]}{\cup} (\partial M \times [-2,\infty]),
\]
and write $\coldoublehat \colon \partial M \times [-2,\infty] \hookrightarrow \doublehat{M}$ for the inclusion of the right-hand side of this pushout. We may extend $e_0$ to an embedding
\[
\doublehat{e}_0 \colon \doublehat{M} \lhto \rinftymt,
\]
defining $\doublehat{e}_0(\coldoublehat(x,t)) = (b(x),t)$ for $(x,t) \in \partial M \times [-2,0]$. There is a diagram of topological spaces
\begin{equation}\label{eq:pullback-diagram}
\centering
\begin{split}
\begin{tikzpicture}
[x=1mm,y=1mm]
\node (tr) at (50,15) {$\mathrm{Emb}(N',\rinftymt) / \mathrm{Diff}_K(N')$};
\node (bl) at (0,0) {$\mathrm{Emb}^c(T,\doublehat{M})/K$};
\node (br) at (50,0) {$\mathrm{Emb}(U,\rinftymt) / K,$};
\draw[->] (tr) to (br);
\draw[->] (bl) to (br);
\end{tikzpicture}
\end{split}
\end{equation}
where $\mathrm{Emb}^c(T,\doublehat{M}) \subset \mathrm{Emb}(T,\doublehat{M})$ is the subspace of embeddings $T \hookrightarrow \doublehat{M}$ such that the image of the zero-section $o(P) \subset T$ is contained in $\hat{M} \subset \doublehat{M}$. The vertical map is given by precomposition by $\tau_j \circ \sigma$ and the horizontal map is given by postcomposition by $\doublehat{e}_0$ (and restriction of the domain).
\end{construction}

\begin{lem}
The diagram \eqref{eq:pullback-diagram} is a diagram of right $G$-spaces with respect to the following well-defined actions. For the bottom two spaces, the action of $g \in G$ is given by choosing any element $h \in H$ such that $z(h)=g$ and acting by precomposition by $h$. For the top-right space, the action is given by first choosing any element $h \in H$ such that $z(h)=g$ and then any diffeomorphism of $N'$ whose restriction to $U$ along the embedding $\tau_j$ is $h$.
\end{lem}
\begin{proof}
The well-defined-ness of the described $G$-actions on the bottom two spaces follows from the fact that we have a short exact sequence
\[
1 \to K \lhto H \xrightarrow{\, z|_H \,} G \to 1,
\]
and the horizontal map of \eqref{eq:pullback-diagram} is clearly equivariant with respect to these actions. By one of our assumptions just before the statement of Theorem \ref{tmain} on page \pageref{tmain}, the map $\mathrm{Diff}_H(N') = \mathrm{Diff}_H(N) \to H$ given by restriction along the embedding $\tau_j$ is surjective. It is therefore also surjective after composing with $z|_H \colon H \to G$, and we have another short exact sequence
\[
1 \to \mathrm{Diff}_K(N') \lhto \mathrm{Diff}_H(N') \longrightarrow G \to 1,
\]
which shows that the described $G$-action on the top-right space of \eqref{eq:pullback-diagram} is well-defined. To see that the vertical map of \eqref{eq:pullback-diagram} is equivariant, we note that the action of $H \leq \mathrm{Diff}_{\mathrm{fib}}(T)$ on $U \subset T$ commutes with the involution $\sigma \colon U \to U$. To see this, note that, under the identification $U \cong \partial T \times (0.5,1)$, the $H$-action is trivial on the second component and the involution $\sigma$ is trivial on the first component.
\end{proof}

\begin{defn}
Let $Z$ be the pullback in the category of topological spaces of the diagram \eqref{eq:pullback-diagram}. Since this is a diagram of right $G$-spaces, $Z$ is also a right $G$-space if we give it the diagonal action. In other words, we take $Z$ to be the pullback in the category of right $G$-spaces.
\end{defn}

\begin{lem}\label{l:path-connected}
The composite map
\[
\pi \colon Z \longrightarrow \mathrm{Emb}^c(T,\doublehat{M})/K \longrightarrow \mathrm{Emb}(P,\hat{M}) = E
\]
is a Serre fibration with path-connected fibres.
\end{lem}
\begin{proof}
By a mild extension of Theorem B of \cite{Palais1960Localtrivialityof} (see also \cite{Cerf1961Topologiedecertains}) allowing manifolds with boundary, the action of $\mathrm{Diff}_\partial(\hat{M})$ on $\mathrm{Emb}(P,\hat{M})$ is locally retractile. Similarly, a mild extension of Proposition 4.15 of \cite{Palmer2018HomologicalstabilitymoduliI} implies that the action of $\mathrm{Diff}_\partial(\rinftymt)$ on $\mathrm{Emb}(U,\rinftymt)/K$ is locally retractile. Theorem A of \cite{Palais1960Localtrivialityof} then implies that the right-hand map $\mathrm{Emb}^c(T,\doublehat{M})/K \to E$ above and the vertical map of \eqref{eq:pullback-diagram} are fibre bundles. The left-hand map $Z \to \mathrm{Emb}^c(T,\doublehat{M})/K$ above is a pullback of the vertical map of \eqref{eq:pullback-diagram}, so it is also a fibre bundle. A composition of two fibre bundles is not necessarily a fibre bundle, but it is at least a Serre fibration.

It is not hard to show that the fibres of the vertical map of \eqref{eq:pullback-diagram} are path-connected, using the fact that we are considering embeddings into infinite-dimensional Euclidean space. The fibres of $Z \to \mathrm{Emb}^c(T,\doublehat{M})/K$ are therefore also path-connected.

Fix an embedding $e \in \mathrm{Emb}(P,\hat{M})$ and denote the fibre of the map $\mathrm{Emb}^c(T,\doublehat{M}) \to \mathrm{Emb}(P,\hat{M})$ over $e$ by $\mathrm{Emb}(T,\doublehat{M})_e$. Note that we do not yet take the quotient by $K$. This is almost the space $\mathrm{Tub}(e)$ of tubular neighbourhoods of $e$. More accurately, there is a fibration $\mathrm{Emb}(T,\doublehat{M})_e \to \mathrm{Aut}(T)$ to the topological group of bundle automorphisms of $T$ (as a disc bundle with structure group $O(m-p)$), and the fibre over the identity is $\mathrm{Tub}(e)$. This may be summarised as follows:
\begin{equation}\label{eq:two-fibration-sequences}
\centering
\begin{split}
\begin{tikzpicture}
[x=1mm,y=1mm]
\node (tl) at (5,12) {$\mathrm{Tub}(e)$};
\node (tm) at (30,12) {$\mathrm{Emb}(T,\doublehat{M})_e$};
\node (tr) at (60,12) {$\mathrm{Emb}^c(T,\doublehat{M})$};
\node (bm) at (30,0) {$\mathrm{Aut}(T)$};
\node (br) at (60,0) {$\mathrm{Emb}(P,\hat{M})$};
\node at (69,0) [anchor=west] {$\ni e$};
\draw[->] (tl) to (tm);
\draw[->] (tm) to (tr);
\draw[->] (tm) to (bm);
\draw[->] (tr) to (br);
\end{tikzpicture}
\end{split}
\end{equation}
Now, the space $\mathrm{Tub}(e)$ of tubular neighbourhoods of $e$ is contractible \cite[Proposition 31]{Godin2007Higherstringtopology}, so the fibres of $\mathrm{Emb}^c(T,\doublehat{M}) \to \mathrm{Emb}(P,\hat{M})$ are homotopy equivalent to $\mathrm{Aut}(T)$. Note that this group is exactly the kernel of the map
\[
z \colon \mathrm{Diff}_{\mathrm{fib}}(T) \longrightarrow \mathrm{Diff}(P)
\]
defined at the beginning of \S\ref{s:symm-diff-groups}. To study the fibres of $\mathrm{Emb}^c(T,\doublehat{M})/K \to \mathrm{Emb}(P,\hat{M})$, we quotient three of the spaces in \eqref{eq:two-fibration-sequences} by the action of $K$:
\begin{equation}\label{eq:two-fibration-sequences-K}
\centering
\begin{split}
\begin{tikzpicture}
[x=1mm,y=1mm]
\node (tl) at (2,12) {$\mathrm{Tub}(e)$};
\node (tm) at (30,12) {$\mathrm{Emb}(T,\doublehat{M})_e/K$};
\node (tr) at (65,12) {$\mathrm{Emb}^c(T,\doublehat{M})/K$};
\node (bm) at (30,0) {$\mathrm{Aut}(T)/K$};
\node (br) at (65,0) {$\mathrm{Emb}(P,\hat{M})$};
\node at (74,0) [anchor=west] {$\ni e$};
\draw[->] (tl) to (tm);
\draw[->] (tm) to (tr);
\draw[->] (tm) to (bm);
\draw[->] (tr) to (br);
\end{tikzpicture}
\end{split}
\end{equation}
As $\mathrm{Tub}(e)$ is contractible, the fibres of $\mathrm{Emb}^c(T,\doublehat{M})/K \to \mathrm{Emb}(P,\hat{M})$ are homotopy equivalent to
\[
\mathrm{Aut}(T) / K = \mathrm{ker}(z) / K.
\]
But we assumed just before the statement of Theorem \ref{tmain} on page \pageref{tmain} that this coset space is path-connected. Thus both of the fibrations $Z \to \mathrm{Emb}^c(T,\doublehat{M})/K \to \mathrm{Emb}(P,\hat{M})$ have path-connected fibres, so the composite fibration $\pi$ also has path-connected fibres.
\end{proof}

Note that, by construction, the map $\pi$ is $G$-equivariant. Recall from \S\ref{s:mod-labelled} (see Input \ref{input-data}) that we also need to choose a basepoint $\bar{\imath}_0 \in Z$ and a continuous homomorphism $\bar{\gamma} \colon \bR \to \mathrm{Homeo}^G(Z)$ in order to define the moduli space of labelled submanifolds $C_{nP}(M,Z;G)$.

Choose an embedding $\upsilon \colon N' \hookrightarrow \bR^\infty \times (-0.5,0.5) \subset \rinftymt$ so that the following diagram commutes:
\begin{center}
\begin{tikzpicture}
[x=1mm,y=1mm]
\node (tl) at (0,12) {$U$};
\node (tm) at (20,12) {$N'$};
\node (bl) at (0,0) {$U$};
\node (bm) at (20,0) {$\doublehat{M}$};
\node (r) at (40,6) {$\rinftymt$};
\draw[<->] (tl) to node[left,font=\small]{$\sigma$} (bl);
\draw[->] (tl) to node[above,font=\small]{$\tau_j|_U$} (tm);
\draw[->] (bl) to node[below,font=\small]{$\tau_i|_U$} (bm);
\draw[->] (tm) to node[above,font=\small]{$\upsilon$} (r);
\draw[->] (bm) to node[below,font=\small]{$\doublehat{e}_0$} (r);
\end{tikzpicture}
\end{center}

\begin{rmk}
The choice of $\upsilon$ can be made independently of $e_0$, since we have prescribed how $e_0$ acts on $\mathrm{col}(\partial M \times [0,1])$, and therefore how $\doublehat{e}_0$ acts on $\coldoublehat(\partial M \times [-2,1])$, and the image $\tau_i(U)$ is contained in $\coldoublehat(\partial M \times (-0.5,0.5))$.
\end{rmk}

Then $([\tau_i],[\upsilon]) \in Z$ and $\pi([\tau_i],[\upsilon]) = i = i_0$. So we may set $\bar{\imath}_0 = ([\tau_i],[\upsilon])$.

We may extend the ``shift'' map of Definition \ref{d:shift} by the identity to a diffeomorphism $\mathrm{sh}_r \colon \doublehat{M} \to \doublehat{M}$ for each $r \in \bR$. We write $\mathrm{id} \times \bar{\theta}(r)$ for the self-diffeomorphism of $\rinftymt = \bR^\infty \times [-2,\infty)$ that is the identity on $\bR^\infty$ and acts by (the restriction of) $\bar{\theta}(r) \colon \bR \to \bR$ on $[-2,\infty)$. With this notation, we define a map $\bar{\gamma} \colon \bR \to \mathrm{Map}(Z,Z)$ by
\[
\bar{\gamma}(r) \colon ([\alpha],[\beta]) \longmapsto ([\mathrm{sh}_r \circ \alpha],[(\mathrm{id} \times \bar{\theta}(r)) \circ \beta]).
\]
One may easily check that $\bar{\gamma}$ is a well-defined, continuous map and that its image lies in $\mathrm{Homeo}^G(Z) \leq \mathrm{Map}(Z,Z)$. It is also a group homomorphism (since $\bar{\theta}$ is) and each $\bar{\gamma}(r)$ covers the self-homeomorphism $\gamma(r) = \mathrm{sh}_r \circ -$ of $E$. This completes the construction of the input data needed (see Input \ref{input-data}) in order to apply Definition \ref{d:labelled-disconn-submfld}.

\begin{defn}
We may now apply Definition \ref{d:labelled-disconn-submfld} to the data $(\pi \colon Z \to E , \bar{\gamma} , \bar{\imath}_0)$ constructed above to obtain spaces $C_{nP}(M,Z;G)$ and $C_{\hat{n}P}(M,Z;G)$, as well as a stabilisation map
\[
s_n \colon C_{nP}(M,Z;G) \longrightarrow C_{\hat{n}P}(M,Z;G).
\]
These may be thought of as \emph{moduli spaces of disconnected submanifolds labelled by parametric-connected-sum-data}.
\end{defn}

In the rest of this subsection, we will show that the moduli space $C_{nP}(M,Z;G)$, for this particular fibration $\pi \colon Z \to E$, is homotopy equivalent to the fibres of the bundle \eqref{eq:main-fibre-bundle}: see Proposition \ref{p:identify-fibres}. First we establish a lemma that we will need in the proof of this proposition.

\begin{lem}\label{l:path-connected-orbit}
The space $\mathrm{Emb}(nT,M)/(H \wr \Sigma_n)$ has a left-action of the group $\mathrm{Diff}_\partial(M)$ of diffeomorphisms of $M$ that act by the identity on a neighbourhood of its boundary. The embedding $\Phi$ from \textup{\S\ref{s:param-connected-sum}} gives us a basepoint $[\Phi]$ for $\mathrm{Emb}(nT,M)/(H \wr \Sigma_n)$. The orbit of this basepoint under the action of $\mathrm{Diff}_\partial(M)$ is path-connected.
\end{lem}

\begin{rmk}\label{r:path-connected-orbit}
The map $[- \circ \Phi] \colon \mathrm{Diff}_\partial(M) \to \mathrm{Emb}(nT,M)/(H \wr \Sigma_n)$ is a fibre bundle, by Propositions 4.15 and 4.7 of \cite{Palmer2018HomologicalstabilitymoduliI}, so its image, the orbit of $[\Phi]$, must be a union of path-components. By Lemma \ref{l:path-connected-orbit} it is exactly the path-component of $\mathrm{Emb}(nT,M)/(H \wr \Sigma_n)$ containing $[\Phi]$.
\end{rmk}

\begin{proof}[Proof of Lemma \ref{l:path-connected-orbit}]
Let $\varphi \in \mathrm{Diff}_\partial(M)$; we will find a path of embeddings $nT \hookrightarrow M$ from $\varphi \circ \Phi$ to $\Phi$. The image of $\Phi$ is contained in a collar neighbourhood of $\partial M$, so we may choose a path $t \mapsto \varphi_t$ in $\mathrm{Diff}_\partial(M)$ so that $\varphi_0 = \varphi$ and $\varphi_1$ restricts to the identity on this collar neighbourhood, in particular on the image of $\Phi$. Then $t \mapsto \varphi_t \circ \Phi$ is the required path of embeddings.
\end{proof}

\begin{proof}[Second proof of Lemma \ref{l:path-connected-orbit}, in a special case]
If the embedding $P \hookrightarrow \partial M$ has image contained in a coordinate neighbourhood $\bR^{m-1} \subseteq \partial M$, then we may give an alternative proof. We start with a certain collection of copies of $T$ embedded into $M$, we are given a diffeomorphism $\varphi$ of $M$ that fixes $\partial M$ and we need to show that there is a path of embeddings taking the images of the copies of $T$ under $\varphi$ back to their original positions. The assumption on the embedding $P \hookrightarrow \partial M$ means that we may assume that the original copies of $T$ are enclosed in pairwise disjoint codimension-zero balls, which reduces the problem to the case where $T$ is a codimension-zero ball: we need to find an isotopy taking each ball back to its original position before we applied the diffeomorphism $\varphi$. We may do this for one ball at a time, so let us assume that there is only one ball. If we denote the inclusion of this ball into $M$ by $b \colon B \hookrightarrow M$, then our task is to find a path of embeddings from $\varphi \circ b$ to $b$. First, since $M$ is path-connected, we may move $\varphi \circ b$ to a new embedding $b'$ such that $b'(0) = b(0)$. We may then modify $b'$ by a rotation near $b'(0)$ to a new embedding $b''$ whose derivative at $0 \in B$ agrees with that of $b$. This is possible using just a rotation, and not a reflection, because either $M$ is orientable, in which case $\varphi$ is an orientation-preserving diffeomorphism (since it fixes $\partial M$, and $\partial M$ is non-empty), or $M$ is non-orientable, in which case we may if necessary push $b'$ around a non-orientable loop in $M$ in order to change the sign of the determinant of its derivative at $0$. Next, we may shrink $b''$ so that $b''(B) \subseteq b(B)$, so now we have a self-embedding $c = b^{-1} \circ b'' \colon B \hookrightarrow B$ and we need to find an isotopy from this to the identity. Note that $c(0) = 0$ and the derivative of $c$ at $0$ is the identity. We may therefore define an isotopy $c_t$ by $c_t(x) = tc(x/t)$ for $t \in (0,1]$ and $c_0 = \mathrm{id}$.
\end{proof}

\begin{prop}\label{p:identify-fibres}
The fibre of \eqref{eq:main-fibre-bundle} over $[e_0] \in \mathrm{Emb}_b(M,\rinftyz) / \mathrm{Diff}_\partial(M)$ is homotopy-equivalent to $C_{nP}(M,Z;G)$. More precisely, there is a canonical inclusion $\Psi^{-1}([e_0]) \hookrightarrow C_{nP}(M,Z;G)$, which is a homotopy equivalence. The corresponding statement for $\hat{\Psi}$ also holds\,\textup{:} Write $\hat{e}_0 = \doublehat{e}_0|_{\hat{M}}$. There is a canonical inclusion $\hat{\Psi}^{-1}([\hat{e}_0]) \hookrightarrow C_{\hat{n}P}(M,Z;G)$, which is a homotopy equivalence.
\end{prop}
\begin{proof}
We will do this in three steps: (1) give an explicit description of $\Psi^{-1}([e_0])$ and note that it is path-connected, (2) give an explicit description of $C_{nP}(M,Z;G)$ and show that it contains a homeomorphic copy of $\Psi^{-1}([e_0])$, and (3) show that the inclusion is a homotopy equivalence.

\textbf{Step 0.}
Before this, however, we recall a basic fact that we will use in the next step. Let $X$ be a left $G$-space, and assume that the $G$-action on $X$ is \emph{locally retractile} (see for example Definition~4.5 of \cite{Palmer2018HomologicalstabilitymoduliI}). Then for any $x \in X$ there is a $G$-equivariant homeomorphism $G/\mathrm{stab}_G(x) \cong \mathrm{orbit}(x)$. To see this, first note that the action map $-\cdot x \colon G \to X$ induces a continuous bijection $G/\mathrm{stab}_G(x) \to \mathrm{orbit}(x) \subseteq X$, which is $G$-equivariant. Then Theorem A of \cite{Palais1960Localtrivialityof} implies that this map is a fibre bundle, in particular an open map, and so it is a homeomorphism.

\textbf{Step 1.}
Rewriting the definition of $X_n$ a little, we may describe it as the subspace of
\[
\mathrm{Pullback}\bigl(\,\mathrm{Emb}_b(M,\rinftyz) \longrightarrow \mathrm{Emb}(nU,\rinftyz) \longleftarrow \mathrm{Emb}(nN',\rinftyz)\,\bigr)
\]
of pairs of embeddings $(e,f)$ such that $f(nN')$ is disjoint from the closure\footnote{We need to take the closure here since $M$ was not assumed to be compact.} of $e(M \smallsetminus \Phi(n\mathring{T}))$. The first map in the pullback diagram is given by restriction along $\Phi$ and the second map is given by restriction along $\tau_j$ followed by the involution $\sigma$ of $U$. The quotient $X_n/\Sigma_H \mathrm{Diff}(M \underset{nP}{\sharp} nN)$ may therefore be described as the subspace of
\[
\mathrm{Pullback}\biggl(\,\frac{\mathrm{Emb}_b(M,\rinftyz)}{\Sigma_H \mathrm{Diff}(M,nT)} \longrightarrow \frac{\mathrm{Emb}(nU,\rinftyz)}{H \wr \Sigma_n} \longleftarrow \frac{\mathrm{Emb}(nN',\rinftyz)}{\mathrm{Diff}_H(N) \wr \Sigma_n}\,\biggr)
\]
of pairs $([e],[f])$ satisfying the same disjointness condition (\cf the pullback square \eqref{eq:pullback-symm-boundary} in \S\ref{s:symm-diff-groups}). The fibre $\Psi^{-1}([e_0])$ is the subspace where the image of $e$ agrees with the image of $e_0$, so it may be described as the subspace of
\[
\mathrm{Pullback}\biggl(\,\frac{\mathrm{Diff}_\partial(M)}{\Sigma_H \mathrm{Diff}(M,nT)} \longrightarrow \frac{\mathrm{Emb}(nU,\rinftyz)}{H \wr \Sigma_n} \longleftarrow \frac{\mathrm{Emb}(nN',\rinftyz)}{\mathrm{Diff}_H(N) \wr \Sigma_n}\,\biggr)
\]
of pairs $([\eta],[f])$ such that $f(nN')$ is disjoint from the closure of $e_0(M \smallsetminus \eta(\Phi(n\mathring{T})))$. A mild extension of Proposition 4.15 of \cite{Palmer2018HomologicalstabilitymoduliI} shows that $\mathrm{Emb}(nT,M)/(H \wr \Sigma_n)$ is $\mathrm{Diff}_\partial(M)$-locally retractile. The stabiliser of $[\Phi] \in \mathrm{Emb}(nT,M)/(H \wr \Sigma_n)$ is the subgroup $\Sigma_H \mathrm{Diff}(M,nT) \leq \mathrm{Diff}_\partial(M)$, so via the ``topological orbit-stabiliser theorem'' (see Step 0 above) we deduce that $\mathrm{Diff}_\partial(M) / \Sigma_H \mathrm{Diff}(M,nT)$ is homeomorphic to the orbit of $[\Phi]$ in $\mathrm{Emb}(nT,M)/(H \wr \Sigma_n)$. By Lemma \ref{l:path-connected-orbit} and Remark \ref{r:path-connected-orbit}, the orbit of $[\Phi]$ in $\mathrm{Emb}(nT,M)/(H \wr \Sigma_n)$ is the path-component of $[\Phi]$ in $\mathrm{Emb}(nT,M)/(H \wr \Sigma_n)$. Thus $\mathrm{Diff}_\partial(M) / \Sigma_H \mathrm{Diff}(M,nT)$ is homeomorphic to the path-component of $[\Phi]$ in $\mathrm{Emb}(nT,M)/(H \wr \Sigma_n)$. We may therefore describe $\Psi^{-1}([e_0])$ as the subspace of
\[
\mathrm{Pullback}\biggl(\,\frac{\mathrm{Emb}(nT,M)}{H \wr \Sigma_n} \longrightarrow \frac{\mathrm{Emb}(nU,\rinftyz)}{H \wr \Sigma_n} \longleftarrow \frac{\mathrm{Emb}(nN',\rinftyz)}{\mathrm{Diff}_H(N) \wr \Sigma_n}\,\biggr)
\]
of pairs $([\Phi'],[f])$ such that $f(nN')$ is disjoint from the closure of $e_0(M \smallsetminus \Phi'(n\mathring{T}))$ and there exists a path $[\Phi'] \rightsquigarrow [\Phi]$ in $\mathrm{Emb}(nT,M)/(H \wr \Sigma_n)$.

It is now not hard to show that the fibre $\Psi^{-1}([e_0])$ is path-connected. Choose a basepoint $([\Phi_0'],[f_0])$ for it and consider any other point $([\Phi'],[f])$. There is a path $[\Phi'] \rightsquigarrow [\Phi_0']$ in the left-hand space of the pullback diagram. The image of this path in the middle space may be lifted to a path $[f] \rightsquigarrow [f_1]$ in the right-hand space, since the right-hand map of the pullback diagram is a fibre bundle, and therefore a Serre fibration. Since we are considering embeddings into $\rinftyz$ we may easily ensure that this path of embeddings satisfies the disjointness condition, at each point in time during the path, with respect to the path $[\Phi'] \rightsquigarrow [\Phi_0']$, so we have a path $([\Phi'],[f]) \rightsquigarrow ([\Phi_0'],[f_1])$ in $\Psi^{-1}([e_0])$. Choose a path $f_1 \rightsquigarrow f_0$ of embeddings $nN' \hookrightarrow \rinftyz$ disjoint from the closure of $e_0(M \smallsetminus \Phi_0'(n\mathring{T}))$ and constant when restricted to $nU$. This gives us a path $([\Phi_0'],[f_1]) \rightsquigarrow ([\Phi_0'],[f_0])$ in $\Psi^{-1}([e_0])$. Thus we have shown that $\Psi^{-1}([e_0])$ is path-connected.

\textbf{Step 2.}
Recall from Definition \ref{d:labelled-disconn-submfld} that $C_{nP}(M,Z;G)$ is a certain path-component of $Z_n / (G \wr \Sigma_n)$. Unravelling this definition for the fibration $\pi \colon Z \to E$ of Lemma \ref{l:path-connected}, we may describe $Z_n$ as the subspace of
\[
\mathrm{Pullback}\biggl(\,\frac{\mathrm{Emb}(T,\doublehat{M})^n}{K^n} \longrightarrow \frac{\mathrm{Emb}(U,\rinftymt)^n}{K^n} \longleftarrow \frac{\mathrm{Emb}(N',\rinftymt)^n}{\mathrm{Diff}_K(N')^n}\,\biggr)
\]
of tuples of embeddings $(([\varphi_1],\ldots,[\varphi_n]),([f_1],\ldots,[f_n]))$ such that each $\varphi_\alpha(P)$ is contained in $\mathring{M}$ and the images $\varphi_1(P),\ldots,\varphi_n(P)$ are pairwise disjoint. The quotient $Z_n / (G \wr \Sigma_n)$ is therefore the subspace of
\[
\mathrm{Pullback}\biggl(\,\frac{\mathrm{Emb}(T,\doublehat{M})^n}{H \wr \Sigma_n} \longrightarrow \frac{\mathrm{Emb}(U,\rinftymt)^n}{H \wr \Sigma_n} \longleftarrow \frac{\mathrm{Emb}(N',\rinftymt)^n}{\mathrm{Diff}_H(N) \wr \Sigma_n}\,\biggr)
\]
of collections of embeddings $(\{[\varphi_1],\ldots,[\varphi_n]\},\{[f_1],\ldots,[f_n]\})$ satisfying the same two conditions. Comparing this with the final description of $\Psi^{-1}([e_0])$ in the previous step, we see that there is a canonical inclusion $\Psi^{-1}([e_0]) \hookrightarrow Z_n / (G \wr \Sigma_n)$. Moreover, $\Psi^{-1}([e_0])$ is path-connected and contains the basepoint configuration $\{[\bar{\imath}_1],\ldots,[\bar{\imath}_n]\}$, so there is in fact a canonical inclusion
\[
\Psi^{-1}([e_0]) \lhto C_{nP}(M,Z;G).
\]

\textbf{Step 3.}
A point in $C_{nP}(M,Z;G)$ lies in the subspace $\Psi^{-1}([e_0])$ if and only if
\begin{itemizeb}
\item[(a)] $\varphi_1(T),\ldots,\varphi_n(T)$ are pairwise disjoint and contained in $M$,
\item[(b)] $f_1(N'),\ldots,f_n(N')$ are pairwise disjoint and contained in $\rinftyz$,
\item[(c)] $\bigcup_\alpha f_\alpha(N')$ is disjoint from the closure of $e_0(M \smallsetminus \bigcup_\beta \varphi_\beta(\mathring{T}))$,
\item[(d)] there is a path in $\mathrm{Emb}(nT,M)/(H \wr \Sigma_n)$ from $\{[\varphi_1],\ldots,[\varphi_n]\}$ to $[\Phi]$.
\end{itemizeb}
In fact, property (d) is automatic once we have property (a). Since $C_{nP}(M,Z;G)$ is path-connected, there is a path in $\mathrm{Emb}(T,\doublehat{M})^n / (H \wr \Sigma_n)$ from $\{[\varphi_1],\ldots,[\varphi_n]\}$ to $[\Phi]$, and the $n$ embedded copies of $P \subset T$ in $\doublehat{M}$ are pairwise disjoint and contained in $\mathring{M}$ at each point in time during this path. We may therefore shrink the tubular neighbourhoods $T \supset P$ by an appropriate amount at each point during the path, to obtain a new path in $\mathrm{Emb}(nT,M)/(H \wr \Sigma_n)$ from $\{[\varphi_1],\ldots,[\varphi_n]\}$ to $[\Phi]$.

We therefore would like to define a deformation retraction that begins with a point in $C_{nP}(M,Z;G)$ and ends with a new point in $C_{nP}(M,Z;G)$ satisfying the disjointness properties (a), (b) and (c). In fact, we will not do this for $C_{nP}(M,Z;G)$, but instead for its ordered analogue $F_{nP}(M,Z;G)$, the covering space in which the $n$ embedded copies of $T$ in $\doublehat{M}$ are equipped with an ordering. If we write $\widetilde{\Psi}^{-1}([e_0]) \subset F_{nP}(M,Z;G)$ for the restriction of this covering space to $\Psi^{-1}([e_0]) \subset C_{nP}(M,Z;G)$, we have a map of fibre sequences:
\begin{center}
\begin{tikzpicture}
[x=1mm,y=1mm]
\node (tl) at (0,12) {$F_{nP}(M,Z;G)$};
\node (tm) at (30,12) {$C_{nP}(M,Z;G)$};
\node (tr) at (60,12) {$B\Sigma_n$};
\node (bl) at (0,0) {$\widetilde{\Psi}^{-1}([e_0])$};
\node (bm) at (30,0) {$\Psi^{-1}([e_0])$};
\node (br) at (60,0) {$B\Sigma_n$};
\draw[->] (tl) to (tm);
\draw[->] (tm) to (tr);
\draw[->] (bl) to (bm);
\draw[->] (bm) to (br);
\incl{(bl)}{(tl)}
\incl{(bm)}{(tm)}
\node at (60,6) {\rotatebox{90}{$=$}};
\end{tikzpicture}
\end{center}
If we can define a deformation retraction for the inclusion $\widetilde{\Psi}^{-1}([e_0]) \subset F_{nP}(M,Z;G)$, then the map of long exact sequences of homotopy groups will imply that the inclusion $\Psi^{-1}([e_0]) \subset C_{nP}(M,Z;G)$ is also a (weak) homotopy equivalence.

We now sketch a deformation retraction that begins with a point in $F_{nP}(M,Z;G)$ and ends with a new point in $F_{nP}(M,Z;G)$ satisfying properties (a), (b) and (c). Since the ``cores'' $\varphi_\alpha(P)$ are pairwise disjoint and contained in $\mathring{M}$, we may shrink the tubular neighbourhoods $T \supset P$ as above to ensure property (a); this may moreover be done in a canonical way, so that the deformation retraction is continuous. To ensure property (c), we may choose $k$ such that $e_0(M) \subseteq \bR^k \subset \rinftyz$ and modify the embedded copies of $N'$ in $\rinftymt$ so that their $(k+1)$-st coordinate is non-zero on $N \smallsetminus \tau_j(T) \subset N'$.

Finally, we need to ensure property (b). We may push $\rinftymt$ into $\rinftyz$ by increasing its last coordinate to ensure that the embedded copies of $N'$ are contained in $\rinftyz$. To ensure that they are pairwise disjoint, we modify them by straight-line homotopies so that, for each $r \in \{1,\ldots,n\}$, the $(k+1+r)$-th coordinate of the $r$-th copy of $N'$ is non-zero on $N \smallsetminus \tau_j(T) \subset N'$, and the $(k+1+r)$-th coordinate of every other copy of $N'$ is zero on $N \smallsetminus V \subset N'$, where $V$ is a very small open neighbourhood of $\tau_j(T)$ in $N$. (This is where we need to use the ordering.) This will force the different copies of $N'$ to be pairwise disjoint, except possibly on the subsets $V \smallsetminus \tau_j(T) \subset N'$. However, we may control these neighbourhoods to be very small, i.e.\ very close to the corresponding $\varphi_r(T)$, so, by shrinking these further if necessary, we may ensure that the different copies of $N'$ are pairwise disjoint everywhere.
\end{proof}

\subsection{A map of fibre bundles}

We now define a continuous map $X_n \to X_{\hat{n}}$, in order to obtain a map of bundles from $\Psi$ to $\hat{\Psi}$.

\begin{defn}\label{d:stn}
Define
\[
\mathrm{st}_n \colon X_n \longrightarrow X_{\hat{n}}
\]
to send a pair of embeddings $(e,f)$ to $(\hat{e},\hat{f})$, where:
\begin{itemizeb}
\item[(i)] $\hat{e} = e$ on $M \subset \hat{M}$ and $\hat{e}(\colhat(x,t)) = (b(x),t)$ for $(x,t) \in \partial M \times [-1,0]$.
\item[(ii)] Recall that
\[
M \underset{\hat{n}P}{\sharp} \hat{n}N \;=\; (\hat{M} \smallsetminus \hat{\Phi}(\hat{n}T')) \,\underset{\hat{n}U}{\cup}\, \hat{n}N'
\]
(\cf Definition \ref{d:param-conn-sum} and Notation \ref{not:nprime}). We define $\hat{f} = \hat{e}$ on the subspace $\hat{M} \smallsetminus \hat{\Phi}(\hat{n}T')$, and also $\hat{f} = f$ on the subspace $nN' \subseteq M \underset{nP}{\sharp} nN$, so it remains to define $\hat{f}$ on $\hat{n}N' \smallsetminus nN' = \{0\} \times N'$. Here we define it by
\[
\{0\} \times N' = N' \xrightarrow{\;\;\upsilon\;\;} \bR^\infty \times (-0.5,0.5) \xrightarrow{\;(\mathrm{id},-0.5)\;} \bR^\infty \times (-1,0) \subset \rinftym.
\]
\end{itemizeb}
It is then an easy exercise to check that $(\hat{e},\hat{f})$ is an element of $X_{\hat{n}}$ and that $\mathrm{st}_n$ is continuous.
\end{defn}

\begin{rmk}
Recall that $X_n$ and $X_{\hat{n}}$ have actions of $\Sigma_H \mathrm{Diff}(M \underset{nP}{\sharp} nN)$ and $\Sigma_H \mathrm{Diff}(M \underset{\hat{n}P}{\sharp} \hat{n}N)$ respectively. Since we defined $\hat{e}$ and $\hat{f}$ above so that $\hat{e} = e$ on $M$ and $\hat{f} = f$ on $M \underset{nP}{\sharp} nN$, it follows that $\mathrm{st}_n$ is equivariant with respect to the top horizontal map (continuous homomorphism) of the diagram \eqref{eq:symdiff-square}, which we now denote by
\[
\sigma_n \colon \Sigma_H \mathrm{Diff}(M \underset{nP}{\sharp} nN) \longrightarrow \Sigma_H \mathrm{Diff}(M \underset{\hat{n}P}{\sharp} \hat{n}N).
\]
It therefore follows from Lemma \ref{l:pbundle} that the induced map
\[
\overline{\mathrm{st}}_n \colon X_n / \Sigma_H \mathrm{Diff}(M \underset{nP}{\sharp} nN) \longrightarrow X_{\hat{n}} / \Sigma_H \mathrm{Diff}(M \underset{\hat{n}P}{\sharp} \hat{n}N)
\]
is a model for $B\sigma_n$, the map induced on classifying spaces by $\sigma_n$. It also fits into a map of bundles
\begin{equation}\label{eq:map-of-bundles}
\centering
\begin{split}
\begin{tikzpicture}
[x=1mm,y=1mm]
\node (tl) at (-30,15) {$X_n / \Sigma_H \mathrm{Diff}(M \underset{nP}{\sharp} nN)$};
\node (tr) at (30,15) {$X_{\hat{n}} / \Sigma_H \mathrm{Diff}(M \underset{\hat{n}P}{\sharp} \hat{n}N)$};
\node (bl) at (-30,0) {$\mathrm{Emb}_b (M,\rinftyz)/\mathrm{Diff}_\partial(M)$};
\node (br) at (30,0) {$\mathrm{Emb}_b (M,\rinftym)/\mathrm{Diff}_\partial(M),$};
\draw[->] ($ (tl.east) + (0,1.5) $) to node[above,font=\small]{$\overline{\mathrm{st}}_n$} ($ (tr.west) + (0,1.5) $);
\draw[->] ($ (tl.south) + (0,2) $) to node[left,font=\small]{$\Psi$} (bl);
\draw[->] ($ (tr.south) + (0,2) $) to node[right,font=\small]{$\hat{\Psi}$} (br);
\draw[->] (bl) to (br);
\end{tikzpicture}
\end{split}
\end{equation}
where the bottom horizontal map is defined by $[e] \mapsto [\hat{e}]$, where $\hat{e}$ is defined as in (i) in Definition \ref{d:stn} above.
\end{rmk}

\begin{lem}\label{l:split-injectivity}
The map $\overline{\mathrm{st}}_n$ induces split-injections on homology in all degrees.
\end{lem}
\begin{proof}
In order to apply Lemma 2 of \cite{Dold1962DecompositiontheoremsSn} to deduce split-injectivity, it suffices to define maps
\[
X_n/D_n \longrightarrow \mathrm{Sp}^{\binom{n}{k}} (X_k / D_k)
\]
satisfying a certain equation up to homotopy, where we are using the temporary abbreviation $D_n = \Sigma_H \mathrm{Diff}(M \underset{nP}{\sharp} nN)$. We briefly sketch how to construct such maps. They may be defined by
\[
[e,f] \,\longmapsto\, \sum_{S \subseteq \{1,\ldots,n\} \, , \, \lvert S \rvert = k} [e_S',f_S'],
\]
where $f_S'$ is the composition
\begin{equation}\label{eq:fsprime}
M \underset{kP}{\sharp} kN \;\cong\; M \underset{S \times P}{\sharp} (S \times N) \lhto M \underset{nP}{\sharp} nN \xrightarrow{\; f \;} \rinftyz
\end{equation}
and $e_S' = e \circ \eta_S$, where $\eta_S \in \mathrm{Diff}_\partial(M)$ corresponds to the left-hand diffeomorphism of \eqref{eq:fsprime} in the sense that these two diffeomorphisms agree on $M \smallsetminus \Phi(kT')$.
\end{proof}

\subsection{Stability for symmetric diffeomorphism groups}

\begin{lem}
The map of base spaces in \eqref{eq:map-of-bundles} is a homotopy equivalence.
\end{lem}
\begin{proof}
In order to define a homotopy inverse, we need a path of compactly-supported embeddings $\gamma_t \colon [-1,\infty) \hookrightarrow [-1,\infty)$ such that $\gamma_0(-1) = 0$, $\gamma_t(0) \geq 0$ for all $t$ and $\gamma_1 = \mathrm{id}$. Then applying this isotopy to the last coordinate of $\rinftym$ determines a homotopy inverse for the bottom horizontal map in \eqref{eq:map-of-bundles}. The finer details of the construction are similar to those of Lemma 5.26 of \cite{Palmer2018HomologicalstabilitymoduliI}.
\end{proof}

Now fix a point $[e_0]$ in the base space of $\Psi$ as in Construction \ref{construction:Z}. There is a commutative square\footnote{In fact this square does not commute on the nose, but if we replace the top horizontal map $s_n$ with a homotopic map $s_n^\prime$, then it does commute on the nose. Recall that $s_n$ adjoins the element $[\bar{\imath}_0] = [[\tau_i],[\upsilon]] \in Z/G$ to an unordered tuple of elements of $Z/G$. The map $s_n^\prime$ instead adjoins the element $[[\tau_i - 0.5],[\upsilon - 0.5]]$, where the ${-}\,0.5$ denotes a shift along the collar neighbourhood of $\doublehat{M}$ (for $\tau_i$) and in the last coordinate of $\rinftymt$ (for $\upsilon$).}
\begin{equation}\label{eq:identifying-fibres}
\centering
\begin{split}
\begin{tikzpicture}
[x=1mm,y=1mm]
\node (tl) at (0,12) {$C_{nP}(M,Z;G)$};
\node (tr) at (40,12) {$C_{\hat{n}P}(M,Z;G)$};
\node (bl) at (0,0) {$\Psi^{-1}([e_0])$};
\node (br) at (40,0) {$\hat{\Psi}^{-1}([\hat{e}_0])$};
\incl{(bl)}{(tl)}
\incl{(br)}{(tr)}
\draw[->] (tl) to node[above,font=\small]{$s_n$} (tr);
\draw[->] (bl) to node[below,font=\small]{$\overline{\mathrm{st}}_n|_{[e_0]}$} (br);
\end{tikzpicture}
\end{split}
\end{equation}
whose vertical maps are homotopy equivalences by Proposition \ref{p:identify-fibres}. The map of bundles \eqref{eq:map-of-bundles} induces a map of Serre spectral sequences, which is an isomorphism on $E^2$ pages up to vertical degree $\tfrac{n}{2}-1$ (or $\tfrac{n}{2}$ if we take field coefficients), by Theorem \ref{tlabelled} and Lemma \ref{l:path-connected}. The Zeeman comparison theorem \cite{Zeeman1957proofofcomparison} (\cf also \cite[Theorem 1.2]{Ivanov1993homologystabilityTeichmuller} or \cite[Remarque 2.10]{CollinetDjamentGriffin2013Stabilitehomologiquepour}) then implies that $\overline{\mathrm{st}}_n = B\sigma_n$ induces isomorphisms on homology up to degree $\tfrac{n}{2}-1$ (or $\tfrac{n}{2}$ with field coefficients). Together with Lemma \ref{l:split-injectivity}, this proves homological stability for the top horizontal map of \eqref{eq:symdiff-square}. By Remark \ref{r:canonical-heq}, a special case of this implies homological stability also for the bottom horizontal map of \eqref{eq:symdiff-square}.\footnote{We need to check that the third assumption stated just before Theorem \ref{tmain} is satisfied for this special case (i.e.\ for this choice of $N$), but this is clear (\cf the last sentence of Remark \ref{r:canonical-heq}).} This completes the proof of Theorem \ref{tmain}, assuming Theorem \ref{tlabelled}.

\section{Diffeomorphism groups of manifolds with conical singularities}\label{s:conical-singularities}

Before proving Theorem \ref{tlabelled}, we first discuss manifolds with (discrete) conical singularities and their diffeomorphism groups. These are a special case of manifolds with \emph{Baas-Sullivan singularities}, introduced in \cite{Sullivan1967Hauptvermutungmanifolds,Baas1973bordismtheorymanifolds}, for which the singular set is not necessarily discrete, but may itself be a smooth manifold of positive dimension (see \S\ref{ss:BS-singularities} for a brief overview). We then prove homological stability for diffeomorphism groups of manifolds with conical singularities, with respect to adding new singularities of a fixed type --- in fact, we will see that this is nothing more than a special case of homological stability for symmetric diffeomorphism groups, already contained in Theorem \ref{tmain}.

\subsection{Manifolds with conical singularities}\label{ss:conical}

Fix a smooth, closed $(m-1)$-dimensional manifold $L$. The \emph{open cone} on $L$ is
\[
\mathrm{cone}(L) = (L \times [0,\infty)) / (L \times \{0\})
\]
and we write $\star = [L \times \{0\}] \in \mathrm{cone}(L)$ for the point at the tip of the cone.

\begin{defn}\label{d:manifold-with-conical-sing}
An \emph{$m$-dimensional smooth manifold with conical $L$-singularities} consists of a space $M$ equipped with a discrete subset $\Sigma \subseteq M$ and the structure of a smooth manifold on $M \smallsetminus \Sigma$. In addition, the points $\sigma \in \Sigma$ are equipped with pairwise disjoint open neighbourhoods $\sigma \in U_\sigma \subseteq M$ and homeomorphisms $U_\sigma \cong \mathrm{cone}(L)$ sending $\sigma$ to $\star$ and restricting to diffeomorphisms $U_\sigma \smallsetminus \{\sigma\} \cong L \times (0,\infty)$.
\end{defn}

For example, when $m=1$, this amounts to a graph of fixed valency $k$ (i.e.\ every vertex has the same valency $k$). In this case $\Sigma$ is the set of vertices of the graph and $L$ is the $0$-manifold $\{1,\ldots,k\}$. Two other examples are as follows.

\begin{eg}
If $M$ is a smooth $m$-dimensional manifold, then we may take $\Sigma = \varnothing$. Alternatively, we may also take $\Sigma$ to be any discrete subset and $L = S^{m-1}$, although in this case we forget the smooth structure at the marked points $\Sigma$.
\end{eg}

\begin{eg}
If $\ell \subset \bR^3$ is a link with $k$ components, then $M = \bR^3 / \ell$ is a manifold with a single conical singularity ($\Sigma$ is the single point $[\ell] \in M$) of type $L = \bigsqcup_k (S^1 \times S^1)$.
\end{eg}

\subsection{Diffeomorphisms of singular manifolds}\label{ss:diff-singular}

Now we also fix a subgroup $H \leq \mathrm{Diff}(L)$. Note that each self-diffeomorphism $\psi$ of $L$ determines a self-homeomorphism $\mathrm{cone}(\psi)$ of $\mathrm{cone}(L)$ via $\mathrm{cone}(\psi)([x,t]) = [\psi(x),t]$.

\begin{defn}\label{d:diff-manifold-with-conical-sing}
Let $M$ be a manifold with conical $L$-singularities. A \emph{diffeomorphism} of $M$ is then a homeomorphism $\varphi \colon M \to M$ such that $\varphi(\Sigma) = \Sigma$ and the restriction $\varphi|_{M \smallsetminus \Sigma}$ is a diffeomorphism. Moreover, for each $\sigma \in \Sigma$ we require that $\varphi(U_\sigma) = U_{\varphi(\sigma)}$ and that the induced homeomorphism $\mathrm{cone}(L) \cong U_\sigma \to U_{\varphi(\sigma)} \cong \mathrm{cone}(L)$ is of the form $\mathrm{cone}(h)$ for some $h \in H$. These form a subgroup
\[
\mathrm{Diff}^L_H(M) \leq \mathrm{Homeo}(M).
\]
\end{defn}

\subsection{Relation to manifolds with Baas-Sullivan singularities}\label{ss:BS-singularities}

Manifolds with discrete conical singularities are a special case ($s=0$) of the more general notion of a \emph{manifold with Baas-Sullivan singularities}, in which the singular set is a smooth $s$-dimensional manifold. Fix a smooth, closed manifold $L$ of dimension $m-s-1$, which will be called the \emph{type}, or \emph{link}, of the singular set.

\begin{defn}
A \emph{manifold with Baas-Sullivan singularities} \cite{Sullivan1967Hauptvermutungmanifolds,Baas1973bordismtheorymanifolds} consists of a topological space $M$ equipped with the following data: a subset $\Sigma \subseteq M$, a structure of a smooth $s$-dimensional manifold (without boundary) on $\Sigma$ and a structure of a smooth $m$-dimensional manifold (possibly with boundary) on $M \smallsetminus \Sigma$, an open neighbourhood $U \supseteq \Sigma$ in $M$ and a homeomorphism
\[
\theta \colon (U,\Sigma) \longrightarrow (\Sigma \times \mathrm{cone}(L) , \Sigma \times \{\star\})
\]
whose restriction to $U \smallsetminus \Sigma \longrightarrow \Sigma \times L \times (0,\infty)$ is a diffeomorphism.
\end{defn}

\begin{defn}\label{d:diff-manifold-with-sing}
Given a manifold $\mathbf{M} = (M,\Sigma,U,\theta)$ with Baas-Sullivan singularities of type $L$, a \emph{diffeomorphism} of $\mathbf{M}$ is a homeomorphism $\varphi \colon M \to M$ fixing $\Sigma$ and $U$ setwise, such that the restrictions $\varphi|_\Sigma$ and $\varphi|_{M \smallsetminus \Sigma}$ are diffeomorphisms and $\varphi|_U = \theta^{-1} \circ (\varphi|_\Sigma \times \mathrm{cone}(\psi)) \circ \theta$ for some diffeomorphism $\psi \colon L \to L$. We may also fix a subgroup $H \leq \mathrm{Diff}(L)$ and require $\psi$ to be an element of $H$, in which case this is an \emph{$H$-diffeomorphism} of $\mathbf{M}$.
\end{defn}

\begin{rmk}
There is a another viewpoint on manifolds with Baas-Sullivan singularities, where, instead of a singular set $\Sigma \subseteq M$ equipped with a conical open neighbourhood, one instead has a smooth manifold $M$ with boundary, equipped with a collar neighbourhood and an embedding $\Sigma \times L \hookrightarrow \partial M$ whose image is a union of components of $\partial M$. A morphism $\varphi \colon M \to M'$ between such objects is then defined to be a smooth map, compatible with the collar neighbourhoods, sending $\Sigma \times L$ to $\Sigma' \times L$, such that the restriction $\varphi|_{\Sigma \times L}$ is a product of smooth maps $\Sigma \to \Sigma'$ and $L \to L$. A diffeomorphism of $M$ is an automorphism in this category. The definitions above are recovered by taking the quotient of $M$ by the equivalence relation corresponding to the partition
\[
\{ \{\sigma\} \times L \mid \sigma \in \Sigma \} \cup \{ \{x\} \mid x \in M \smallsetminus (\Sigma \times L) \}.
\]
In particular, the conical neighbourhood of $\Sigma$ is the image of the collar neighbourhood under this quotient. For more details, see \cite{Botvinnik1992Manifoldswithsingularities} or \cite{Perlmutter2015Cobordismcategorymanifolds}. (Note: our definition of a diffeomorphism of a manifold with Baas-Sullivan singularities is a mild generalisation of that of \cite[Definition 3.1]{Perlmutter2015Cobordismcategorymanifolds}, where the restriction $\varphi|_{\Sigma \times L}$ is required to be the product of a smooth map $\Sigma \to \Sigma'$ and the identity $L \to L$. In Definition \ref{d:diff-manifold-with-sing} this corresponds to $H$-diffeomorphisms of $\mathbf{M}$ with $H = \{\mathrm{id}\} \leq \mathrm{Diff}(L)$.)
\end{rmk}

\begin{rmk}
The definition of manifolds with Baas-Sullivan singularities may be generalised to allow a collection of smooth, closed manifolds $L_k$ of dimension $m-s-1$. In the case where $s=0$ (corresponding to Definition \ref{d:manifold-with-conical-sing}), this amounts to saying that each open neighbourhood $U_\sigma$ should be identified with $\mathrm{cone}(L_k)$ for some $k$. See \cite{Botvinnik1992Manifoldswithsingularities} or \cite{Perlmutter2013CobordismCategoryof} for more details.
\end{rmk}

\subsection{Relation to symmetric diffeomorphism groups}\label{ss:rel-symm-diff}

We now construct a specific sequence $\mathbf{M}_n$ of manifolds with $n$ conical singularities of a fixed type.

\begin{defn}\label{d:sequence-Mn}
As in \S\ref{s:param-connected-sum}, we fix an embedding $i \colon P \hookrightarrow \partial M \subseteq \hat{M}$, a metric on the normal bundle $\nu_i \to P$ of $i$ and a tubular neighbourhood $T = D(\nu_i) \hookrightarrow \partial M \times (-\tfrac12,\tfrac12) \subseteq \hat{M}$. As described in \S\ref{s:param-connected-sum}, this induces an embedding
\[
\Phi \colon nT = \{1,\ldots,n\} \times T \lhto \mathring{M}.
\]
Recall that $\mathring{T}$ denotes the interior of $T$ and $T' \subset T$ denotes the closed sub-disc-bundle of radius $\tfrac12$.

Given these inputs, we construct $\mathbf{M}_n$, a manifold with $n$ conical $\partial T$-singularities, as follows: starting with the manifold $M$, collapse the subset $\Phi(\{k\} \times T')$ to a point $\sigma_k$, for each $k \in \{1,2,\ldots,n\}$. Then the singularity set is $\Sigma = \{\sigma_1,\sigma_2,\ldots,\sigma_n\}$ and each point $\sigma_k$ is equipped with a conical neighbourhood $U_{\sigma_k} \cong \mathrm{cone}(\partial T)$ given by the image of $\Phi(\{k\} \times \mathring{T})$ under the collapse map $M \to \mathbf{M}_n$.
\end{defn}

\begin{rmk}
As a space, $\mathbf{M}_n$ is homeomorphic to the quotient of $M_\Phi$ obtained by collapsing each $\Phi(\{k\} \times \partial T) \subset \partial M_\Phi$ to a point --- see Remark \ref{r:mphi}.
\end{rmk}

For a smooth fibre bundle $\xi \colon E \to B$ with structure group $G$, write $\mathrm{Diff}_{\mathrm{fib}}(E) \leq \mathrm{Diff}(E)$ for the subgroup of diffeomorphisms $\varphi$ that respect the partition of $E$ into fibres of $\xi$ (it then follows that $\xi \circ \varphi = \bar{\varphi} \circ \xi$ for some diffeomorphism $\bar{\varphi}$ of $B$) and that act on fibres by elements of $G$. In particular we may apply this definition to the disc and sphere bundles (\cf \S\ref{s:symm-diff-groups})
\[
\xi \colon T \longrightarrow P \qquad\text{and}\qquad \xi_\partial = \xi|_{\partial T} \colon \partial T \longrightarrow P,
\]
whose structure groups are both $O(m-p)$. Since an element of $O(m-p)$ is determined by its action on $S^{m-p-1} = \partial D^{m-p}$, the restriction map $\mathrm{Diff}_{\mathrm{fib}}(T) \to \mathrm{Diff}_{\mathrm{fib}}(\partial T)$ is an isomorphism, and we will identify these groups via this isomorphism.

\begin{lem}\label{l:relate-sing-symm}
There is a natural isomorphism
\[
H_* (\mathrm{Diff}_H^{\partial T}(\mathbf{M}_n)) \;\cong\; H_* (\Sigma_H \mathrm{Diff}(M,nT))
\]
for any subgroup $H \leq \mathrm{Diff}_{\mathrm{fib}}(T) = \mathrm{Diff}_{\mathrm{fib}}(\partial T) \leq \mathrm{Diff}(\partial T)$.
\end{lem}
\begin{proof}
This follows directly from the construction of $\mathbf{M}_n$, unravelling Definitions \ref{d:symm-diff-group} and \ref{d:diff-manifold-with-conical-sing}.
\end{proof}

Now Theorem \ref{tmain} and Lemma \ref{l:relate-sing-symm} immediately imply:

\begin{acoro}\label{coro-singularities}
Suppose that $p \leq \tfrac12(m-3)$ and the subgroup $H \leq \mathrm{Diff}_{\mathrm{fib}}(T) \leq \mathrm{Diff}(\partial T)$ has been chosen so that condition \textup{(b)} of \textup{\S\ref{s:hs}} holds. Then there are isomorphisms
\[
H_* (\mathrm{Diff}_H^{\partial T}(\mathbf{M}_n)) \;\cong\; H_* (\mathrm{Diff}_H^{\partial T}(\mathbf{M}_{n+1}))
\]
for $* \leq \tfrac{n}{2} - 1$, and for $* \leq \tfrac{n}{2}$ if we take field coefficients.
\end{acoro}

\section{Twisted homological stability}\label{s:ths}

We will prove Theorem \ref{tlabelled} in \S\ref{s:proof-labelled} as a corollary of a twisted homological stability theorem for moduli spaces of disconnected submanifolds, which we prove in this section. This is a direct analogue of the main result of \cite{Palmer2018Twistedhomologicalstability}, which deals with configuration spaces of points, so we will not include all possible details in this section, since most of the constructions and proofs go through verbatim as in \cite{Palmer2018Twistedhomologicalstability}, with just minor changes of notation.

We return to the setup of \S\ref{s:mod-labelled}, but without any labels for now. So we have a smooth, connected $m$-dimensional manifold $M$ with non-empty boundary, a collar neighbourhood $\mathrm{col} \colon \partial M \times [0,\infty] \hookrightarrow M$ and an embedding $i \colon P \hookrightarrow \partial M$. Recall also that we have extended $M$ by lengthening its collar:
\[
\hat{M} = M \underset{\partial M \times [0,\infty]}{\cup} (\partial M \times [-1,\infty]),
\]
and in Definition \ref{d:shift} we constructed ``shifted'' embeddings $i_r \colon P \hookrightarrow \hat{M}$ for any $r \in \bR$, where $i_0 = i$.

Fix a subgroup $G \leq \mathrm{Diff}(P)$ and write $E = \mathrm{Emb}(P,\hat{M})$. Then in Definition \ref{d:labelled-disconn-submfld} we constructed moduli spaces $C_{nP}(M;G) \subseteq \mathrm{Sp}^n(E/G)$ and $C_{\hat{n}P}(M;G) \subseteq \mathrm{Sp}^{n+1}(E/G)$ and a stabilisation map
\begin{equation}\label{eq:stab-map-recalled}
s_n \colon C_{nP}(M;G) \longrightarrow C_{\hat{n}P}(M;G) \cong C_{(n+1)P}(M;G)
\end{equation}
between them.

\begin{rmk}\label{r:nhat-nplusone}
The identification on the right-hand side of \eqref{eq:stab-map-recalled} is given by identifying the interior $M(-1)$ of $\hat{M}$ (\cf Notation \ref{not:Mr}) with the interior $M(0)$ of $M$ via a diffeomorphism supported on the collar neighbourhood $\mathrm{col}(\partial M \times [-1,\infty]) \subset \hat{M}$. It will be more convenient in this section to think of the target of the stabilisation map as $C_{(n+1)P}(M;G)$. There is a natural basepoint $\{[i_1],\ldots,[i_n]\}$ of $C_{nP}(M;G)$, and $s_n$ is basepoint-preserving.
\end{rmk}

\begin{defn}\label{d:bpm}
The category $\cB_P(M)$ associated to these data has non-negative integers as objects, and a morphism $m \to n$ is given by a choice of $k \leq \mathrm{min}(m,n)$ and a path $\ell$ in $C_{kP}(M;G)$ with $\ell(0) \subseteq \{ [i_1],\ldots,[i_m] \}$ and $\ell(1) \subseteq \{ [i_1],\ldots,[i_n] \}$, up to endpoint-preserving homotopy. The identities are given by constant paths and composition is defined by concatenation of paths, as well as forgetting any ``strand'' of a path that does not match up with a ``strand'' of the other path, analogous to the composition of partially-defined functions. See (2) on page 151 of \cite{Palmer2018Twistedhomologicalstability} for an illustration. The construction in \S 3.1 of \cite{Palmer2018Twistedhomologicalstability} extends directly to this setting, and equips this category with an endofunctor
\[
S \colon \cB_P(M) \longrightarrow \cB_P(M),
\]
whose effect on objects is $n \mapsto n+1$, together with a natural transformation $\iota \colon \mathrm{id}_{\cB_P(M)} \to S$. Any functor $T \colon \cB_P(M) \to \mathsf{Ab}$ to the category of abelian groups (or any abelian category) may then be given a \emph{degree} by defining $\mathrm{deg}(0) = -1$ and recursively
\[
\mathrm{deg}(T) = \mathrm{deg}(\mathrm{coker}(T \to T \circ S)) + 1
\]
for $T \neq 0$. The automorphism group of $n$ in $\cB_P(M)$ is $\pi_1(C_{nP}(M;G))$, so there are well-defined twisted homology groups $H_*(C_{nP}(M;G);T(n))$. The homomorphism $T(\iota_n) \colon T(n) \to T(n+1)$ is equivariant with respect to the map induced on $\pi_1$ by \eqref{eq:stab-map-recalled}, so there are induced maps on twisted homology
\begin{equation}\label{eq:map-on-twisted-homology}
H_*(C_{nP}(M;G);T(n)) \longrightarrow H_*(C_{(n+1)P}(M;G);T(n+1)).
\end{equation}
\end{defn}

\begin{athm}\label{ttwisted}
If $p \leq \tfrac12(m-3)$, $G$ is an open subgroup of $\mathrm{Diff}(P)$ and $T \colon \cB_P(M) \to \mathsf{Ab}$ is a functor of degree $d < \infty$, then \eqref{eq:map-on-twisted-homology} is split-injective in all degrees, and an isomorphism for $* \leq \tfrac{n-d}{2}$.
\end{athm}

The proof of Theorem \ref{ttwisted} is a direct generalisation of \S 3 and \S 6 of \cite{Palmer2018Twistedhomologicalstability}, so we will just sketch the steps involved. Fix a functor $T \colon \cB_P(M) \to \mathsf{Ab}$.

\begin{defn}\label{d:endred}
For $S \subseteq \{1,\ldots,n\}$ let $f_S \colon n \to n$ be the morphism of $\cB_P(M)$ given by the constant path in $C_{(n - \lvert S \rvert)P}(M)$ at the point $\{ [i_s] \mid s \in \{1,\ldots,n\} \smallsetminus S \}$. Then $T(f_S)$ is an endomorphism of $T(n)$ and we may define a subgroup
\[
T_n^k = \mathrm{im}(T(f_{\{1,\ldots,n-k\}})) \; \cap \!\! \bigcap_{i=n-k+1}^n \mathrm{ker}(T(f_{\{i\}}))
\]
of $T(n)$ for any $0 \leq k \leq n$. We write
\begin{equation}\label{eq:covering-space}
C_{(n-k,k)P}(M;G) \longrightarrow C_{nP}(M;G)
\end{equation}
for the covering space in which $k$ copies of $P$ are coloured red and the remaining $n-k$ copies are coloured green. Equivalently, this is the connected covering space of $C_{nP}(M;G)$ corresponding to the subgroup
\[
\mathrm{end}_n^{-1}(\Sigma_{n-k} \times \Sigma_k) \subseteq \pi_1(C_{nP}(M;G)),
\]
where $\mathrm{end}_n \colon \pi_1(C_{nP}(M;G)) \to \Sigma_n$ is the homomorphism that remembers just the permutation of $\{[i_1],\ldots,[i_n]\}$ induced by a path. We also write
\begin{equation}\label{eq:forget-green}
\mathrm{red}_{(n-k,k)} \colon C_{(n-k,k)P}(M;G) \longrightarrow C_{kP}(M;G)
\end{equation}
for the map that forgets all green parts of a configuration.
\end{defn}

Proposition 3.5 and Lemma 6.4 of \cite{Palmer2018Twistedhomologicalstability} generalise directly to the following.

\begin{prop}\label{p:decomposition}
Each $T_n^k$ is invariant under the action of $\pi_1(C_{(n-k,k)P}(M;G))$ on $T(n)$, and therefore gives a twisted coefficient system for $C_{(n-k,k)P}(M;G)$. The pullback of the coefficient system $T_k^k$ along the map \eqref{eq:forget-green} is naturally isomorphic to $T_n^k$. There is a natural isomorphism of $\bZ[\pi_1(C_{nP}(M;G))]$-modules
\[
T(n) \;\cong\; \bigoplus_{k=0}^n \biggl( \bZ[\pi_1(C_{nP}(M;G))] \underset{\bZ[\pi_1(C_{(n-k,k)P}(M;G))]}{\otimes} T_n^k \biggr).
\]
\end{prop}

Lemma 3.16 of \cite{Palmer2018Twistedhomologicalstability} also generalises directly:

\begin{lem}\label{l:degree-height}
We have $\mathrm{deg}(T) \leq d$ if and only if $T_n^k = 0$ for all $n \geq 0$ and all $k > d$.
\end{lem}

The final lemma that we will need before proving Theorem \ref{ttwisted} is the following.

\begin{lem}\label{l:forget-fibre-bundle}
The map \eqref{eq:forget-green} is a fibre bundle.
\end{lem}
\begin{proof}
By Proposition 4.15 of \cite{Palmer2018HomologicalstabilitymoduliI}, the action of $\mathrm{Diff}_c(\mathring{M})$ on $\mathrm{Emb}(kP,\mathring{M})/(G \wr \Sigma_n)$ is locally retractile. By Lemma 4.6(i) of \cite{Palmer2018HomologicalstabilitymoduliI}, the restriction of this action to the path-component of the identity $\mathrm{Diff}_c(\mathring{M})_0$ is also locally retractile. This restricted action obviously fixes setwise the path-component $C_{kP}(M;G) \subseteq \mathrm{Emb}(kP,\mathring{M})/(G \wr \Sigma_n)$, so by Lemma 4.6(iii) of \cite{Palmer2018HomologicalstabilitymoduliI}, the action of $\mathrm{Diff}_c(\mathring{M})_0$ on $C_{kP}(M;G)$ is locally retractile. Hence Theorem A of \cite{Palais1960Localtrivialityof} implies that \eqref{eq:forget-green} is a fibre bundle.
\end{proof}

\begin{proof}[Proof of Theorem \ref{ttwisted}]
The idea is exactly the same as on pages 172--173 of \cite{Palmer2018Twistedhomologicalstability}, so we just give a sketch of how to adapt it. By Proposition \ref{p:decomposition}, Lemma \ref{l:degree-height} and Shapiro's lemma for covering spaces (see Lemma 6.1 of \cite{Palmer2018Twistedhomologicalstability}) there are natural isomorphisms
\begin{equation}\label{eq:first-step}
H_*(C_{nP}(M;G);T(n)) \;\cong\; \bigoplus_{k=0}^d H_*(C_{(n-k,k)P}(M;G);T_n^k).
\end{equation}
It therefore suffices to show that the lift of the stabilisation map
\begin{equation}\label{eq:lifted-stab}
C_{(n-k,k)P}(M;G) \longrightarrow C_{(n+1-k,k)P}(M;G)
\end{equation}
that adds a new green copy of $P$ to the configuration induces isomorphisms on twisted homology with respect to the coefficient systems $T_n^k$ and $T_{n+1}^k$ up to degree $\tfrac{n-k}{2}$. This stabilisation map is a map of fibre bundles (by Lemma \ref{l:forget-fibre-bundle}) over the space $C_{kP}(M;G)$. Our default basepoint of this space is $\{[i_1],\ldots,[i_k]\}$, but for the next argument it will be more convenient to choose a different basepoint $\{[i'_1],\ldots,[i'_k]\}$, where each embedding $i'_{\alpha} \colon P \hookrightarrow M$ has image disjoint from the image of the collar neighbourhood. We may then define
\[
M' = M \smallsetminus \bigcup_{\alpha = 1}^k i'_{\alpha}(P)
\]
and take the same collar neighbourhood for $M'$ as for $M$. The restriction of the map \eqref{eq:lifted-stab} to the fibres over $\{[i'_1],\ldots,[i'_k]\}$ is the stabilisation map
\[
C_{(n-k)P}(M';G) \longrightarrow C_{(n+1-k)P}(M';G).
\]
There is a subtlety in this statement: it is not hard to see that there are topological embeddings
\begin{center}
\begin{tikzpicture}
[x=1mm,y=1mm]
\node (tl) at (0,12) {$\mathrm{red}_{(n-k,k)}^{-1}(\{ [i'_1],\ldots,[i'_k] \})$};
\node (tr) at (60,12) {$\mathrm{red}_{(n+1-k,k)}^{-1}(\{ [i'_1],\ldots,[i'_k] \})$};
\node (bl) at (0,0) {$C_{(n-k)P}(M';G)$};
\node (br) at (60,0) {$C_{(n+1-k)P}(M';G)$};
\incl{(bl)}{(tl)}
\incl{(br)}{(tr)}
\draw[->] (tl) to (tr);
\draw[->] (bl) to (br);
\end{tikzpicture}
\end{center}
making the square commute, defined by adjoining $\{[i'_1],\ldots,[i'_k]\}$ to a configuration. It remains to see that they are surjective: this follows from Proposition 5.10 of \cite{Palmer2018HomologicalstabilitymoduliI}. Using this identification and the first part of Proposition \ref{p:decomposition}, we therefore have a map of twisted Serre spectral sequences (\cf Proposition 5.7 of \cite{Palmer2018Twistedhomologicalstability}), which is as follows on the $E^2$ pages:
\[
H_s(C_{kP}(M;G);H_t(C_{(n-k)P}(M';G);T_k^k)) \longrightarrow H_s(C_{kP}(M;G);H_t(C_{(n+1-k)P}(M';G);T_k^k)),
\]
where $T_k^k$ is a \emph{constant} coefficient system for the fibres, and which converges to the map on twisted homology induced by \eqref{eq:lifted-stab} with respect to the coefficient systems $T_n^k$ and $T_{n+1}^k$. By Theorem~A of \cite{Palmer2018HomologicalstabilitymoduliI} and the universal coefficient theorem, the map of $E^2$ pages is an isomorphism for $t \leq \tfrac{n-k}{2}$. The Zeeman comparison theorem therefore implies that the map in the limit is also an isomorphism up to degree $\tfrac{n-k}{2}$. This completes the proof of Theorem \ref{ttwisted}, except for the split-injectivity statement.

The proof of split-injectivity in \S 7 of \cite{Palmer2018Twistedhomologicalstability} generalises verbatim to establish the split-injectivity statement of Theorem \ref{ttwisted}.
\end{proof}

\section{Stability for moduli spaces of labelled disconnected submanifolds}\label{s:proof-labelled}

We now prove Theorem \ref{tlabelled} as a corollary of Theorem \ref{ttwisted}. This will be another spectral sequence comparison argument, using a map of Serre spectral sequences induced by the square \eqref{eq:square-stabilisation}, so as a first step we prove:

\begin{lem}\label{l:square-fibre-bundles}
The vertical maps in the square \eqref{eq:square-stabilisation} are Serre fibrations.
\end{lem}
\begin{proof}
We will show that the map $C_{nP}(M,Z;G) \to C_{nP}(M;G)$ is a Serre fibration; an identical argument will then show that the other vertical map of \eqref{eq:square-stabilisation} is also a Serre fibration.

By assumption (see Input \ref{input-data}), the map $\pi \colon Z \to E$ is a Serre fibration and also $G$-equivariant. The $n$-fold product $\pi^n \colon Z^n \to E^n$ is also a Serre fibration, and so is its pullback $\pi_n \colon Z_n \to E_n$ along the inclusion $E_n = \mathrm{Emb}(nP,\mathring{M}) \subset E^n = \mathrm{Emb}(P,\hat{M})^n$. Since the map $\pi_n \colon Z_n \to E_n$ is also $(G \wr \Sigma_n)$-equivariant, there is an induced square
\begin{equation}\label{eq:pullback-square}
\centering
\begin{split}
\begin{tikzpicture}
[x=1mm,y=1mm]
\node (tl) at (0,14) {$Z_n$};
\node (tr) at (30,14) {$Z_n / (G \wr \Sigma_n)$};
\node (bl) at (0,0) {$E_n$};
\node (br) at (30,0) {$E_n / (G \wr \Sigma_n),$};
\draw[->] (tl) to (tr);
\draw[->] (bl) to (br);
\draw[->] (tl) to node[left,font=\small]{$\pi_n$} (bl);
\draw[->] (tr) to node[right,font=\small]{$\bar{\pi}_n$} (br);
\node at (4,10) {$\lrcorner$};
\end{tikzpicture}
\end{split}
\end{equation}
which is a pullback as indicated. By Propositions 4.15 and 4.8 of \cite{Palmer2018HomologicalstabilitymoduliI}, the bottom horizontal map is a principal $(G \wr \Sigma_n)$-bundle. So we know that the left-hand vertical map $\pi_n$ and the bottom horizontal map in \eqref{eq:pullback-square} are Serre fibrations, and the bottom horizontal map is also obviously surjective. Thus Lemma 4.19 of \cite{Palmer2018HomologicalstabilitymoduliI} implies that the right-hand vertical map $\bar{\pi}_n$ is also a Serre fibration. Finally, the map $C_{nP}(M,Z;G) \to C_{nP}(M;G)$ is just the restriction of $\bar{\pi}_n$ to one path-component of its source and one path-component of its target, so it is also a Serre fibration.
\end{proof}

Let $R$ be a ring. There is an induced map of Serre spectral sequences, converging to
\[
H_*(C_{nP}(M,Z;G);R) \longrightarrow H_*(C_{\hat{n}P}(M,Z;G);R)
\]
and whose map of $E^2$ pages is of the form
\[
H_s(C_{nP}(M;G);H_t(f_n^{-1}(i_{\{1,\ldots,n\}});R)) \longrightarrow H_s(C_{\hat{n}P}(M;G);H_t(f_{\hat{n}}^{-1}(i_{\{0,\ldots,n\}});R)),
\]
where $f_n$ and $f_{\hat{n}}$ denote the vertical maps in the square \eqref{eq:square-stabilisation} and where $i_{\{1,\ldots,n\}} = \{ [i_1],\ldots,[i_n] \}$ and $i_{\{0,\ldots,n\}} = \{ [i_0] , [i_1],\ldots,[i_n] \}$ are the basepoints.

\begin{rmk}\label{r:identifications}
In this section (as in \S\ref{s:ths}, see Remark \ref{r:nhat-nplusone}) it will be more convenient to view the targets of the stabilisation maps in \eqref{eq:square-stabilisation} as $C_{(n+1)P}(M,Z;G)$ and $C_{(n+1)P}(M;G)$ respectively, so the map of Serre spectral sequences is then
\begin{equation}\label{eq:stab-map-labelled}
H_*(C_{nP}(M,Z;G);R) \longrightarrow H_*(C_{(n+1)P}(M,Z;G);R)
\end{equation}
in the limit and
\begin{equation}\label{eq:twisted-stab-map}
H_s(C_{nP}(M;G);H_t(f_n^{-1}(i_{\{1,\ldots,n\}});R)) \longrightarrow H_s(C_{(n+1)P}(M;G);H_t(f_{n+1}^{-1}(i_{\{1,\ldots,n+1\}});R))
\end{equation}
on the $E^2$ pages.

For these identifications, we use modifications of the maps $\gamma(1) = \mathrm{sh}_1 \circ - \colon E \to E$ and $\bar{\gamma}(1) \colon Z \to Z$ (\cf Definition \ref{d:shift}). Namely, we choose a diffeomorphism $\kappa \colon \hat{M} \to M$ (note that there is no $\hat{}$ on the codomain) so that $\kappa = \mathrm{sh}_1$ on $M \subset \hat{M}$. This induces a $G$-equivariant endomorphism $\kappa \circ - \colon E \to E$, where we recall that $E = \mathrm{Emb}(P,\hat{M})$. We then choose a $G$-equivariant lift $\bar{\kappa} \colon Z \to Z$ of this so that $\bar{\kappa} = \bar{\gamma}(1)$ on $\pi^{-1}(\mathrm{Emb}(P,M)) \subset Z$. Then the identifications
\[
\phantom{Z;Z;Z;}C_{\hat{n}P}(M;G) \cong C_{(n+1)P}(M;G) \qquad\qquad C_{\hat{n}P}(M,Z;G) \cong C_{(n+1)P}(M,Z;G)
\]
are defined by
\[
\{[\varphi_0],\ldots,[\varphi_n]\} \longmapsto \{[\kappa \circ \varphi_0],\ldots,[\kappa \circ \varphi_n]\} \qquad\qquad \{[z_0],\ldots,[z_n]\} \longmapsto \{[\bar{\kappa}(z_0)],\ldots,[\bar{\kappa}(z_n)]\}
\]
respectively.
\end{rmk}

\begin{prop}\label{p:twisted-example}
Suppose that $R$ is a principal ideal domain and that the fibration $\pi \colon Z \to E$ has path-connected fibres, whose homology is a flat $R$-module in each degree. Then for each $t \geq 0$ there is a functor $T_t \colon \cB_P(M) \to \mathsf{Ab}$ of degree at most $t$ such that, up to isomorphism, the map $T_t(\iota_n) \colon T_t(n) \to T_t(n+1)$ is the map
\begin{equation}\label{eq:map-of-fibres}
H_t(f_n^{-1}(i_{\{1,\ldots,n\}});R) \longrightarrow H_t(f_{n+1}^{-1}(i_{\{1,\ldots,n+1\}});R)
\end{equation}
induced by the restriction of the top horizontal map of \eqref{eq:square-stabilisation}.
\end{prop}
\begin{proof}
As in \cite{Palmer2018Twistedhomologicalstability}, write $\Sigma$ for the category with objects $\{1,\ldots,n\}$ for non-negative integers $n$ (with $n=0$ corresponding to the empty set) and whose morphisms are partially-defined injections. This may be viewed as a special case of $\cB_P(M)$: for example, $\cB_{\mathrm{pt}}(\bR^3) \cong \Sigma$.

The map $\pi \colon Z \to E$ is $G$-equivariant; write $\bar{\pi} \colon Z/G \to E/G$ for the induced map of orbit spaces. Let $Y$ be the fibre $\bar{\pi}^{-1}([i_0])$ with basepoint $[\bar{\imath}_0]$. Example 4.1 of \cite{Palmer2018Twistedhomologicalstability} gives us a functor
\[
T_t \colon \Sigma \longrightarrow \mathsf{Ab}
\]
such that $T_t(\iota_n) \colon T_t(n) \to T_t(n+1)$ is the map on $H_t$ induced by the inclusion of pointed spaces $Y^n \hookrightarrow Y \times Y^n = Y^{n+1}$, in other words, the map $([\bar{\imath}_0],-,\ldots,-)$. By Lemma 4.2 and Remark 4.4 of \cite{Palmer2018Twistedhomologicalstability} this functor has degree at most $t$. There is a functor $\cB_P(M) \to \Sigma$ given by remembering just the partial injection $\{[i_1],\ldots,[i_m]\} \dashrightarrow \{[i_1],\ldots,[i_n]\}$ induced by a path $\ell$ of configurations as in Definition \ref{d:bpm}. Precomposition by this functor preserves the degree of functors into the category $\mathsf{Ab}$,\footnote{See \S 4.3 of \cite{Palmer2017comparisontwistedcoefficient} for a more general discussion of when precomposition by a functor preserves degree.} so the composition
\[
T_t \colon \cB_P(M) \longrightarrow \Sigma \longrightarrow \mathsf{Ab}
\]
also has degree at most $t$.

The map \eqref{eq:map-of-fibres} is induced by the composition
$f_n^{-1}(i_{\{1,\ldots,n\}}) \to f_{\hat{n}}^{-1}(i_{\{0,\ldots,n\}}) \to f_{n+1}^{-1}(i_{\{1,\ldots,n+1\}})$, which may equivalently be written
\[
\prod_{\alpha = 1}^n \bar{\pi}^{-1}([i_\alpha]) \longrightarrow \prod_{\alpha = 0}^n \bar{\pi}^{-1}([i_\alpha]) \longrightarrow \prod_{\alpha = 1}^{n+1} \bar{\pi}^{-1}([i_\alpha]),
\]
where the first map is $([\bar{\imath}_0],-,\ldots,-)$ and the second is a restriction of $\doublebar{\gamma}(1)^{n+1}$ (\cf Remark \ref{r:identifications}), where $\doublebar{\gamma}(r)$ denotes the map $Z/G \to Z/G$ induced by the $G$-equivariant map $\bar{\gamma}(r) \colon Z \to Z$. The domain may be identified with $Y^n$ via the homeomorphism $\doublebar{\gamma}(1) \times \cdots \times \doublebar{\gamma}(n)$ and the codomain with $Y^{n+1}$ via the homeomorphism $\doublebar{\gamma}(1) \times \cdots \times \doublebar{\gamma}(n+1)$. Under these identifications (using the fact that $\bar{\gamma}$ is a homomorphism), we see that \eqref{eq:map-of-fibres} becomes the map induced on $H_t$ by the inclusion $Y^n \hookrightarrow Y \times Y^n$, which is exactly $T_t(\iota_n)$, as required.
\end{proof}

\begin{proof}[Proof of Theorem \ref{tlabelled}]
The argument in \S 5.2 of \cite{Palmer2018HomologicalstabilitymoduliI} for the split-injectivity part of the statement generalises verbatim to the setting of moduli spaces of disconnected submanifolds with labels. All one needs, in order to apply Lemma 2 of \cite{Dold1962DecompositiontheoremsSn} to deduce split-injectivity, is to be able to define maps
\[
C_{nP}(M,Z;G) \longrightarrow \mathrm{Sp}^{\binom{n}{k}} (C_{kP}(M,Z;G))
\]
satisfying a certain equation up to homotopy. Viewing $C_{nP}(M,Z;G)$ as a subspace of the symmetric power $\mathrm{Sp}^n(Z/G)$ (\cf Definition \ref{d:labelled-disconn-submfld}), we construct such maps as restrictions of the maps
\[
\mathrm{Sp}^n(Z/G) \longrightarrow \mathrm{Sp}^{\binom{n}{k}} (\mathrm{Sp}^k (Z/G))
\]
that forget $n-k$ points in all possible ways. Thus \eqref{eq:stab-map-labelled} is always split-injective.

It remains to prove the second part of the statement, that when $\pi \colon Z \to E$ has path-connected fibres, $p \leq \tfrac12(m-3)$ and $G$ is an open subgroup of $\mathrm{Diff}(P)$, the map \eqref{eq:stab-map-labelled} is an isomorphism for $* \leq \tfrac{n}{2} - 1$, and also for $* \leq \tfrac{n}{2}$ if $R$ is a field.

First let $R$ be a field, so that every $R$-module is flat. Then these three assumptions, together with Theorem \ref{ttwisted} and Proposition \ref{p:twisted-example}, imply that the map \eqref{eq:twisted-stab-map} is an isomorphism for $s \leq \tfrac{n-t}{2}$, in particular for total degree $s+t \leq \tfrac{n}{2}$. The Zeeman comparison theorem then implies that \eqref{eq:stab-map-labelled} is also an isomorphism for $* \leq \tfrac{n}{2}$.

In general, if a continuous map $X \to Y$ induces isomorphisms on homology up to degree $i$ with all field coefficients, then it induces isomorphisms on integral homology (and therefore with any untwisted coefficients, by the universal coefficient theorem) up to degree $i-1$. This follows from the five-lemma applied to the natural long exact sequences induced by the short exact sequences
\[
0 \to \bZ/(p^r) \to \bZ/(p^{r+1}) \to \bZ/(p) \to 0 \qquad\quad 0 \to \bZ \to \bQ \to \bQ/\bZ \cong\! \bigoplus_{p \text{ prime}} \underset{r \to \infty}{\mathrm{colim}} \, \bZ/(p^r) \to 0
\]
of coefficient groups. Thus the statement in the special case when $R$ is a field implies the statement for general $R$.
\end{proof}


\phantomsection
\addcontentsline{toc}{section}{References}
\renewcommand{\bibfont}{\normalfont\small}
\setlength{\bibitemsep}{0pt}
\printbibliography

\vspace{1em}

\noindent {\itshape Institutul de Matematică Simion Stoilow al Academiei Române,
21 Calea Griviței, 010702 București, Romania}

\noindent {\tt mpanghel@imar.ro}

\end{document}